[2021/04/21]
\documentclass[article]{amsart}
\usepackage{amsmath}
\usepackage{amssymb}
\usepackage{amsthm}
\usepackage{amscd}
\usepackage{amsfonts}
\usepackage{amssymb}
\usepackage{mathrsfs}
\usepackage{geometry}
\usepackage{enumerate}
\usepackage{xcolor}
\newgeometry{top=2cm,bottom=2cm,left=2cm,right=2cm}

\allowdisplaybreaks[1]

\newtheorem{thm}{Theorem}[section]
\newtheorem{cor}[thm]{Corollary}
\newtheorem{lem}[thm]{Lemma}

\newtheorem{prop}[thm]{Proposition}

\newtheorem{defn}[thm]{Definition}
\numberwithin{equation}{section}
\newtheorem{rem}{Remark}[section]
\numberwithin{equation}{section}

\let\norm=\enVert

\def\ol#1{\overline{#1}}

\def\wave#1{\widetilde{#1}}
\def\bs#1{\boldsymbol{#1}}
\def\lcase#1{\MakeLowercase{#1}}

\newcommand{\D}{\mathcal{D}}

\newcommand{\US}{\mathfrak{A}}

\newcommand{\fcZ}{\mathcal{Z}}

\def\f#1{\mathfrak{#1}}
\newcommand{\fS}{\f{S}}
\newcommand{\CL}{\mathcal{C}}

\newcommand{\ZZ}{\mathbb{Z}}
\newcommand{\QQ}{\mathbb{Q}}
\newcommand{\ZG}{{\mathbb{Z}}_2}
\newcommand{\NN}{\mathbb{N}}

\newcommand{\RR}{R}

\newcommand{\Qv}{\QQ(v)}
\newcommand{\wQv}{\wave{{\QQ}(v)}}

\def\MN(#1){ M_{#1}(\mathbb{N})}
\def\MNR(#1,#2){ M_{#1}(\mathbb{N})_{#2}}
\def\MNS(#1){ M_{#1}(\mathbb{N})^{\pm}}
\def\MZ(#1){ M_{#1}(\ZG)}
\def\NZ(#1){ (\NN|\ZG)^{#1}}
\def\NZST(#1,#2,#3){ (\NN|\ZG)^{#1}_{#2|#3}}
\def\NZS(#1,#2){ {\NN}^{#1}_{#2}}
\def\MNZ(#1,#2){ M_{#1}(\NN | \ZG)_{#2}}
\def\MNZN(#1){ M_{#1}(\NN | \ZG)}
\def\CMNZ(#1,#2,#3){\Lambda(#1,#2|#3)}
\def\CMN(#1,#2){\Lambda(#1,#2)}
\def\CMNP(#1,#2){\Lambda_{#1,#2}}
\def\MNZNS(#1){ \MNZN(#1)^{\pm}}

\def\SE#1{{#1}^{\ol{0}}}
\def\SO#1{{#1}^{\ol{1}}}
\def\DSE#1{{#1}_{\ol{0}}}

\def\SEE#1{{#1}^{\ol{0}}}
\def\SOE#1{{#1}^{\ol{1}}}
\def\SUP#1{\SE{#1}|\SO{#1}}
\def\SS(#1,#2){{#1}^{\ol{#2}}}
\def\SSE(#1,#2){{#1}^{\ol{#2}}}

\def\Qqs(#1,#2){ \mathcal{Q}({#1},{#2}) }

\newcommand{\Heck}{\mathcal{H}_{r,\RR}}
\newcommand{\HCR}{\mathcal{H}^c_{r,\RR}}

\newcommand{\Qqnr}{\mathcal{Q}_{\lcase{q}}{(\lcase{n},\lcase{r})}}

\def\Qvs(#1){\mathcal{Q}_{\lcase{v}}{(\lcase{#1})}}

\newcommand{\SQvnR}{{\widetilde{\mathcal{Q}}}_{\lcase{v}}{(\lcase{n})}}

\newcommand{\SQvnrR}{{\widetilde{\mathcal{Q}}}_{\lcase{v}}(\lcase{n},\lcase{r})}

\newcommand{\SQvnrRZ}{{\SQvnrR}_{\fcZ}}

\newcommand{\USnv}{{\US[n]}_{\lcase{v}}}

\newcommand{\ep}{\epsilon}

\def\qn{\mathfrak{\lcase{q}}_n}

\def\Uqn{U(\mathfrak{\lcase{q}}_n)}

\def\Uvqn{U_{\lcase{v}}(\mathfrak{\lcase{q}}_{n})}

\def\Uvq(#1){U_{\lcase{v}}(\mathfrak{\lcase{q}}_{\lcase{#1}})}

\newcommand{\Uv}{U_{\lcase{v}}}

\def\USN(#1){{\US[#1]}_{v}}

\newcommand{\dij}{\delta_{i,j}}
\def\SABJR(#1,#2,#3,#4){({#1}|{#2})[\bs{#3}, #4]}
\def\SABJRS(#1,#2,#3,#4){({#1}|{#2})[#3, #4]}
\def\SABJS(#1,#2,#3){({#1}|{#2})[#3]}
\def\SAJRS(#1,#2,#3){{#1}[#2, #3]}
\def\SAJS(#1,#2){{#1}[#2]}
\def\SABJ(#1,#2,#3){({#1}|{#2})[\bs{#3}]}
\def\SAJR(#1,#2,#3){{#1}[\bs{#2}, #3]}
\def\SAJ(#1,#2){{#1}[\bs{#2}]}
\def\ABJR(#1,#2,#3,#4){({#1}|{#2})(\bs{#3}, #4)}
\def\ABJRS(#1,#2,#3,#4){({#1}|{#2})(#3, #4)}
\def\ABJS(#1,#2,#3){({#1}|{#2})(#3)}
\def\AJRS(#1,#2,#3){{#1}(#2, #3)}
\def\AJS(#1,#2){{#1}(#2)}
\def\ABJ(#1,#2,#3){({#1}|{#2})(\bs{#3})}
\def\AJR(#1,#2,#3){{#1}(\bs{#2}, #3)}
\def\AJ(#1,#2){{#1}(\bs{#2})}

\def\snorm#1{\norm{#1}}
\def\STDUE(#1,#2){({#1}+E_{{#2},{#2}+1}-E_{{#2}+1,{#2}+1}|0)}
\def\STDUO(#1,#2){({#1}-E_{{#2}+1,{#2}+1}|E_{{#2},{#2}+1})}
\def\STDLE(#1,#2){({#1}-E_{{#2},{#2}}+E_{{#2}+1,{#2}}|0)}
\def\STDLO(#1,#2){({#1}-E_{{#2},{#2}}|E_{{#2}+1,{#2}})}
\def\STDDE(#1){(D_{#1}|0)}
\def\STDDO(#1,#2){({#1}-E_{{#2},{#2}}|E_{{#2},{#2}})}

\newcommand{\tspan}{\mathrm{span}}

\newcommand{\End}{\mathrm{End}}

\newcommand{\ro}{\mathrm{ro}}
\newcommand{\co}{\mathrm{co}}

\def\SSTEP#1{ {[{#1}]}_{{v}} }
\def\STEP#1{ {[\![{#1}]\!]}_{{q}} }
\def\STEPP#1{{[\![{#1}]\!]}_{{q}^2}}
\def\STEPPD#1{{[\![{#1}]\!]}_{{q},{q}^2}}
\def\STEPPDR#1{{[\![{#1}]\!]}_{{q}^2,{q}}}

\def\VSTEP#1{ {[\![{#1}]\!]}_{{v}^2} }
\def\VSTEPP#1{{[\![{#1}]\!]}_{{v}^4}}

\def\ARGSTEP(#1,#2){ {[\![{#1}]\!]}_{{#2}} }
\def\ARGSSTEP(#1,#2){ {[{#1}]}_{{#2}} }

\newcommand{\where}{\ \bs{|} \ }

\def\wt{\mathrm{wt}}

\newcommand{\spaceintv}{\vspace{1cm}}
\def\intd(#1,#2,#3){\left[\begin{matrix}{#1};{#2}\\{#3}\end{matrix}\right]}
\def\intds(#1,#2){\left[\begin{matrix}{#1}\\{#2}\end{matrix}\right]}
\def\intdss(#1,#2){\intds({#1},{#2})}

\def\parity#1{p({#1})}
\def\homog#1{h({#1})}

\def\TAIJ(#1,#2){ T_{({#1}, {#2})} }
\def\TDIJ(#1,#2){ T_{({#1}, {#2})} }
\def\SAIJ(#1,#2){ S_{({#1}, {#2})} }
\def\SDIJ(#1,#2){ S_{({#1}, {#2})} }
\def\rmText#1{}
\def\rmForm#1{}

\def\AK(#1,#2){ {\overleftarrow{a}}^{#2}_{#1} }
\def\BK(#1,#2){ {\overrightarrow{a}}^{#2}_{#1} }

\def\MNZADD(#1, #2){\SE{#1} + {#2}|\SO{#1}}

\def\TF(#1){ t{({#1})} }
\def\HF(#1){ H{({#1})} }

\newcommand{\cft}{t}

\def\genE{{\mathsf{E}}}
\def\genF{{\mathsf{F}}}
\def\genK{{\mathsf{K}}}

\def\Ad{{A^{\!\star}}}
\def\Bd{{B^{\!\star}}}
\def\Xd{{X^{\!\star}}}
\def\Norm{\partial}

\newcommand{\refdot}{\bs{\cdot}}

\counterwithin{equation}{thm}
\title{The Regular Representation  of the  twisted queer $q$-Schur Superalgebra} 
\author{ Zhenhua Li}
\address{Zhenhua Li,  
College of Information Science and Electronic Engineering, Zhejiang University, Hangzhou,  310058, China
}
\email{zhen-hua.li@qq.com}

\begin{document}
\maketitle

\tableofcontents

\spaceintv
\begin{abstract}
We study the representation theory of the quantum queer superalgebra  ${U_{\lcase{v}}(\mathfrak{\lcase{q}}_{n})}$ 
and obtain some properties of the highest weight modules. 
Furthermore, 
based on the realization of  ${U_{\lcase{v}}(\mathfrak{\lcase{q}}_{n})}$,   we study the representation theory of the twisted queer $q$-Schur superalgebra ${{\widetilde{\mathcal{Q}}}_{\lcase{v}}(\lcase{n},\lcase{r})}$,
and obtain the decomposition of its regular module  as a direct sum of irreducible submodules,
which also means ${{\widetilde{\mathcal{Q}}}_{\lcase{v}}(\lcase{n},\lcase{r})}$ is semisimple.

\end{abstract}
\spaceintv
\section{Introduction}\label{sec_introduction}

In the field of algebraic representation theory,
 the study of quantum algebras and their representations has always been a dynamic and productive research area.

Since quantum groups appeared in the mid-1980s, the family of Schur algebras has grown to include $q$-Schur algebras, affine $q$-Schur algebras, and their superalgebra versions, as seen in \cite{DJ}, \cite{DDF}, \cite{DR} and \cite{DW2018}. 
Quantum Schur superalgebras generalize the classical Schur algebras and their quantum analogs by incorporating superalgebra structures.
The study of quantum Schur superalgebras has emerged as a significant area of research within the broader context of quantum algebra and representation theory, 
see \cite{DGW,DR,EK,DW2018, DW2015, DGLW2024}. 
In particular, 
El Turkey and Kujawa provided a presentation of the Schur superalgebra and its quantum analog in \cite{EK}. 
Similarly, Du and Wan provided a presentation of the queer Schur superalgebra and its quantum analog $\Qqnr$ in \cite{DW2015}, 
which is identified as the endomorphism algebra of a direct sum of certain ``$q$-permutation supermodules''
 for the Hecke-Clifford superalgebra $\HCR$. It provides a bridge between the quantum queer superalgebra 
 and the combinatorial aspects of the symmetric group and Clifford algebras, 
opening up new avenues for exploring their connections. 
Based on the research of \cite{DW2015} and \cite{DW2018}, the authors in \cite{DGLW2024} introduced the twisted queer q-Schur superalgebra $\SQvnrR$, 
which could be viewed as a twisted version of $\Qqnr$.

Quantum Schur superalgebras play a crucial role in understanding the representation theory of quantum groups and Lie superalgebras. 
For example, Du and Wan provided a constructible classication of simple polynimial representations
 for queer $q$-Schur suprealgebra over a certain extension of the field of complex rational functions in \cite{DW2018}, 
and the authors in \cite{DGLW2024} investigated the matrix representation of the regular module
 of queer $q$-Schur suprealgebra $\Qqnr$, which generalized the work of \cite{DW2018}. 

The regular representation of a quantum Schur superalgebra 
can be viewed as the algebra acting on itself by left multiplication. 
This representation is not only a powerful tool for studying the algebra's internal structure but also serves as a bridge connecting the algebra to other areas of mathematics, 
such as combinatorics and algebraic geometry.
Recent developments in the study of quantum Schur superalgebras have
 highlighted the importance of understanding their regular representations, 
as these representations provide a rich source of information about the algebra's irreducible modules and their multiplicities, see \cite{DGZ}, \cite{DGLW2024}. In this paper,
 the focus is on the regular representation of the twisted queer $q$-Schur superalgebra $\SQvnrR$, 
 which  is shown to have a module decomposition.

 Olshanski constructed the quantum queer superalgebra $\Uvqn$,
  which is the quantum deformation of the universal enveloping algebra of the queer Lie superalgebra $\qn$, 
  and established the Schur-Olshanski duality (cf. \cite{Ol}). Based on the description of Olshanski,
   the authors of \cite{GJKK} gave a new presentation of $\Uvqn$ in terms of generators and relations, 
   and initiated the study of highest weight $\Uvqn$-modules
    and laid the foundations of the crystal bases theory for the tensor modules of $\Uvqn$.

Highest weight modules are central to quantum group representation theory,
 serving as the primary tool for classifying finite-dimensional irreducible representations, 
 which are parameterized by dominant weights analogous to classical Lie algebras. 
 They underpin the complete reducibility of modules when ${q}$ is not a root of unity, 
 bridging quantum and classical theories through a natural specialization to $q = 1$. 
 These modules also facilitate connections with other fields like quantum invariant theory  and physics. 
 Their structure and properties enable the transfer of classical techniques to quantum settings, 
 making them indispensable for understanding the algebraic, geometric, 
 and physical implications of quantum groups.
 For more details, see \cite[Section 10.1]{CP}.

 In \cite{GJKK},
  the authors use triangular decomposition of {$\Uvqn$}
   and representation theory of quantum Clifford superalgebras 
   to study highest weight modules over {$\Uvqn$}, 
   proving a classical limit theorem (showing correspondence to classical Lie superalgebra modules) 
   and complete reducibility in the category {${\mathcal{O}}^{\ge 0}_{q} $}. 
 This establishes a theoretical bridge between quantum and classical theories, 
 enabling systematic classification of irreducible modules and laying groundwork for applications in quantum physics and algebra.

In this paper, 
based on the realization of $\Uvqn$, 
we study the highest weight modules over $\Uvqn$,
and focuse on the representation theory of  $\Uvqn$ and  $\SQvnrR$.
We  provide  crucial insights into the structures and properties of these algebras and their modules.
For example, any highest weight module of highest weight $\omega$ is shown to be isomorphic to a quotient of the Verma module $V(\omega)$. 
Finite dimensional irreducible weight modules are characterized,
 and it has been proved that they are highest weight modules with highest weights in a specific form.

We organize this paper as follows. We start with recalling the realizations of $\Uvqn$ and $\Uvqn$ and some important  formulas for $\Uvqn$ in Section \ref{sec_realization}. In Section \ref{sec_modules}, 
we define weight modules, highest weight vectors, and highest weight modules for $\Uvqn$. The decomposition of  $\SQvnrR$ as a $U_{v}$-module is investigated in Section \ref{regular}. And we prove that any irreducible submodules of $\SQvnrR$ are polynomial weight modules and highest weight modules.
Finally, in the last section,  we give a complete decomposition to determine  the structure of the highest weight spaces of these irreducible submodules and provide a complete decomposition of the regular representation of $\SQvnrR$ based on the parity of certain indices.
   
   Throughout the paper, let {$n , r \ge 2$} be positive integers,
  and  $v$ and $q$ are indeterminates such that ${v}^2 = {q}$, and {$\fcZ = \ZZ[{v}, {v}^{-1}]$}.
  
\spaceintv
\section{The realization of {$\Uvqn$}}\label{sec_realization}

\subsection{The queer superalgebra and the quantum queer superalgebra}

\ 

In this section,  we recall the realization of {$\Uqn$} in \cite{GLL},
and  the realization of {$\Uvqn$} in \cite{DGLW2025}.

The associative superalgebra of type Q is defined as 
\begin{align}
	Q_{n}(\QQ) = \{
	\left[\begin{matrix}
		C&D\\
		D&C
	\end{matrix}\right] 
	 \in M_{2n}(\QQ)
	 \where 
	 C,D \in M_n(\QQ)
	\} \subset M_{n|n}(\QQ),
	\label{qmatrix}
\end{align}
	where {$\SE{Q_{n}(\QQ)}$} consists of all matrices satisfying $D=0$,
	and {$\SO{Q_{n}(\QQ)}$} consists of all matrices satisfying $C=0$.
	
For any associative superalgebra {$A$},
let {$ \homog{A} = \SE{A} \cup \SO{A}$} be the set of  homogenouse elements of $A$.
For any {$z \in \SS(A, i)$},  
we say the parity of {$z$} is $i$,
denoted as {$\parity{z} = i$}.
If we define the superbracked as:
\begin{equation}\label{def_superbracked}
\begin{aligned}
[x, y] = x y - {(-1)}^{\parity{x} \cdot \parity{y}} y x ,
\end{aligned}
\end{equation}
where {$x, y \in \homog{A}$},
then   $A$ becomes a Lie superalgebra, which is denoted as  {$A^{-}$}.

The queer Lie superalgebra, denoted by $\qn$, 
is the Lie superalgebra arised from {$Q_{n}(\QQ) $}.
We fix $\mathfrak{h}$ to be the standard Cartan subalgebra of $\qn$ consisting of
matrices of the form ~(\ref{qmatrix})~ with
$C$ being arbitrary diagonal and $D=0$.
Observe that the even subalgebra {$\DSE{\mathfrak{h}}$} of $\mathfrak{h}$ can be identified with the standard Cartan subalgebra of $\mathfrak{gl}(n)$.
Let  {$\{ \bs{\ep}_i \where   1 \le i \le n\}$} be the basis for {$\DSE{\mathfrak{h}}^*$} 
and we define a bilinear form $(\cdot,\cdot)$ on {$\DSE{\mathfrak{h}}^*$}
via
\begin{equation}\label{form}
(\bs{\ep}_i,\bs{\ep}_j):=\bs{\ep}_j(h_i)=\delta_{ij}.
\end{equation}
Then with respect to {$\SE{\mathfrak{h}}$}, the queer Lie superalgebra $\qn$ admits the root system
$$
\Phi=\{\alpha_{i,j}:=\bs{\ep}_i-\bs{\ep}_j|1\le i\not=j\le n\},
$$
where $\alpha_i=\alpha_{i,i+1}$ are the simple roots.

We denote
\begin{align*}
	\bs{Q} = \bigoplus_{j=1}^{n-1} \ZZ {\alpha}_j, \qquad
	\bs{Q}^+ = \bigoplus_{j=1}^{n-1} \NN  {\alpha}_j.	
\end{align*}

\begin{defn}
(See \cite[Section 1]{ GJKK} or \cite{LS})
The queer superalgebra {$\Uqn$} over the field {$\QQ$}  is a superalgebra  generated by 
even elements {$h_{i},  e_j,  f_{j} $}  and odd elements {$h_{\ol{i}}, e_{\ol{j}}, f_{\ol{j}} $}
for  {$1 \le i \le n, 1 \le j \le n-1$},
subjecting  to the relations 
\begin{align*}
(QS1) \qquad
&	[h_{i}, h_{j}] =0, \qquad [h_{i}, h_{\ol{j}}] = 0, 
	\qquad [h_{\ol{i}}, h_{\ol{j}}] = \dij 2 h_{i}; \\
(QS2) \qquad 
	&[h_{i}, e_{j}] = (\bs{\ep}_i,  \alpha_j) e_{j} , \qquad
	[f_{i}, e_{\ol{j}}] = (\bs{\ep}_i,  \alpha_j) e_{\ol{j}} , \\
	&[h_{i}, f_{j}] = -(\bs{\ep}_i,  \alpha_j) f_{j} , \qquad
	[h_{i}, f_{\ol{j}}] = -(\bs{\ep}_i,  \alpha_j) f_{\ol{j}} ; \\
(QS3)  \qquad
&[h_{\ol{i}}, e_{j} ] = (\bs{\ep}_i,  \alpha_j)	e_{\ol{j}},  \qquad
[h_{\ol{i}}, f_{j}] = -(\bs{\ep}_i,  \alpha_j) f_{\ol{j}}, \\
&[h_{\ol{i}}, e_{\ol{j}}] =
\left\{
\begin{aligned}
	&e_{j}, \quad \mbox{ if }\quad i=j \mbox{ or } j+1,  \\
	& 0, \quad \mbox{ otherwise },
\end{aligned}
\right.  \qquad
[h_{\ol{i}}, f_{\ol{j}}] =
\left\{
\begin{aligned}
	&f_{j}, \quad \mbox{ if }\quad i=j \mbox{ or } j+1,  \\
	& 0, \quad \mbox{ otherwise };
\end{aligned}
\right. \\
(QS4)  \qquad
	&[e_{i}, f_{j}] = \dij (h_{i} - h_{i+1}), \qquad
	[e_{\ol{i}}, f_{\ol{j}}] = \dij (h_{i} + h_{i+1}), \\
	&[e_{\ol{i}}, f_{j}] = \dij (h_{\ol{i}} - h_{\ol{i+1}}), \qquad
	[e_{i}, f_{\ol{j}}] = \dij (h_{\ol{i}} - h_{\ol{i+1}}); \\
(QS5) \qquad 
	&[e_{i}, e_{\ol{j}}] = [e_{\ol{i}}, e_{\ol{j}}] = [f_{i}, f_{\ol{j}}] = [f_{\ol{i}}, f_{\ol{j}}] = 0, \qquad \mbox{ where } |i-j| \ne 1, \\
	&[e_{i}, e_{j}] = [f_{i}, f_{j}] = 0, \qquad \mbox{ where } |i-j| > 1, \\
	&[e_{i}, e_{i+1}] = [e_{\ol{i}}, e_{\ol{i+1}}] ,
	\qquad [e_{i}, e_{\ol{i+1}}] = [e_{\ol{i}}, e_{i+1}] ,\\
	&[f_{i+1}, f_{i}] = [f_{\ol{i+1}}, f_{\ol{i}}],
	\qquad [f_{i+1}, f_{\ol{i}}] = [f_{\ol{i+1}}, f_{i}]; \\
(QS6) \qquad 
	&[e_{i}, [e_{i}, e_{j}]] = [e_{\ol{i}}, [e_{i}, e_{j}]] = 0,\qquad
 	[f_{i}, [f_{i}, f_{j}]] = [f_{\ol{i}}, [f_{i}, f_{j}]] = 0,
	\quad \mbox{ where } |i-j| = 1.
\end{align*}
\end{defn}

In \cite[Section 4]{Ol}, Olshanski constructed a quantum deformation $\Uvqn$ of the enveloping algebra $\Uqn$ of the queer Lie superalgebra $\qn$ using a modification of the Reshetikhin-Takhtajan-Faddeev method. The following definition of quantum queer superalgebra is from 
 \cite[Proposition 5.2]{DW2015}.

\begin{defn}\label{defqn}
(See \cite{DW2015})
The quantum queer supergroup $\Uvqn$ is the superalgebra
over $\Qv$  generated by
even generators  {${\genK}_{i}$}, {${\genK}_{i}^{-1}$},  {${\genE}_{j}$},  {${\genF}_{j}$},
and odd generators  {${\genK}_{\ol{i}}$},  {${\genE}_{\ol{j}}$}, {${\genF}_{\ol{j}}$},
for   {$ 1 \le i \le n$}, {$ 1 \le j \le n-1$}, subject to the following relations:
\begin{align*}
({\rm QQ1})\quad
&	{\genK}_{i} {\genK}_{i}^{-1} = {\genK}_{i}^{-1} {\genK}_{i} = 1,  \qquad
	{\genK}_{i} {\genK}_{j} = {\genK}_{j} {\genK}_{i} , \qquad
	{\genK}_{i} {\genK}_{\ol{j}} = {\genK}_{\ol{j}} {\genK}_{i}, \\
&	{\genK}_{\ol{i}} {\genK}_{\ol{j}} + {\genK}_{\ol{j}} {\genK}_{\ol{i}}
	= 2 {\delta}_{i,j} \frac{{\genK}_{i}^2 - {\genK}_{i}^{-2}}{{v}^2 - {v}^{-2}}; \\
({\rm QQ2})\quad
& 	{\genK}_{i} {\genE}_{j} = {v}^{\bs{\ep}_i \refdot \alpha_j} {\genE}_{j} {\genK}_{i}, \qquad
	{\genK}_{i} {\genE}_{\ol{j}} = {v}^{\bs{\ep}_i \refdot \alpha_j} {\genE}_{\ol{j}} {\genK}_{i}, \\
& 	{\genK}_{i} {\genF}_{j} = {v}^{-(\bs{\ep}_i \refdot \alpha_j)} {\genF}_{j} {\genK}_{i}, \qquad
	{\genK}_{i} {\genF}_{\ol{j}} = {v}^{-(\bs{\ep}_i \refdot  \alpha_j)} {\genF}_{\ol{j}} {\genK}_{i}; \;
	(\mbox{here }\alpha_j=\bs\ep_j-\bs\ep_{j+1})\\
({\rm QQ3})\quad
& {\genK}_{\ol{i}} {\genE}_{i} - {v} {\genE}_{i} {\genK}_{\ol{i}} = {\genE}_{\ol{i}} {\genK}_{i}^{-1}, \qquad
	{v} {\genK}_{\ol{i}} {\genE}_{i-1} -  {\genE}_{i-1} {\genK}_{\ol{i}} = - {\genK}_{i}^{-1} {\genE}_{\ol{i-1}}, \\
& {\genK}_{\ol{i}} {\genF}_{i} - {v} {\genF}_{i} {\genK}_{\ol{i}} = - {\genF}_{\ol{i}} {\genK}_{i}, \qquad
	{v} {\genK}_{\ol{i}} {\genF}_{i-1} -  {\genF}_{i-1} {\genK}_{\ol{i}} = {\genK}_{i} {\genF}_{\ol{i-1}},\\
& {\genK}_{\ol{i}} {\genE}_{\ol{i}} + {v} {\genE}_{\ol{i}} {\genK}_{\ol{i}} = {\genE}_{i} {\genK}_{i}^{-1}, \qquad
	{v} {\genK}_{\ol{i}} {\genE}_{\ol{i-1}} +  {\genE}_{\ol{i-1}} {\genK}_{\ol{i}} =   {\genK}_{i}^{-1} {\genE}_{i-1}, \\
& {\genK}_{\ol{i}} {\genF}_{\ol{i}} + {v} {\genF}_{\ol{i}} {\genK}_{\ol{i}} =   {\genF}_{i} {\genK}_{i}, \qquad
	{v} {\genK}_{\ol{i}} {\genF}_{\ol{i-1}} +  {\genF}_{\ol{i-1}} {\genK}_{\ol{i}} = {\genK}_{i} {\genF}_{i-1}, \\
& 
 {\genK}_{\ol{i}} {\genE}_{j} - {\genE}_{j} {\genK}_{\ol{i}} =  {\genK}_{\ol{i}} {\genF}_{j} - {\genF}_{j} {\genK}_{\ol{i}}
 = {\genK}_{\ol{i}} {\genE}_{\ol{j}} + {\genE}_{\ol{j}} {\genK}_{\ol{i}} =  {\genK}_{\ol{i}} {\genF}_{\ol{j}} + {\genF}_{\ol{j}} {\genK}_{\ol{i}}
	= 0, \mbox{ for } j \ne i, i-1; \\
({\rm QQ4}) \quad
& {\genE}_{i} {\genF}_{j} - {\genF}_{j} {\genE}_{i}
	= \delta_{i,j}  \frac{{\genK}_{i} {\genK}_{i+1}^{-1} - {\genK}_{i}^{-1}{\genK}_{i+1}}{{v} - {v}^{-1}}, \\
&
{\genE}_{\ol{i}} {\genF}_{\ol{j}} + {\genF}_{\ol{j}} {\genE}_{\ol{i}}
	= \delta_{i,j}  ( \frac{{\genK}_{i} {\genK}_{i+1} - {\genK}_{i}^{-1} {\genK}_{i+1}^{-1}}{{v} - {v}^{-1}}
	 + ({v} - {v}^{-1}) {\genK}_{\ol{i}} {\genK}_{\ol{i+1}} )  ,\\
& {\genE}_{i} {\genF}_{\ol{j}} - {\genF}_{\ol{j}} {\genE}_{i}
	= \delta_{i,j}  ( {\genK}_{i+1}^{-1} {\genK}_{\ol{i}}  - {\genK}_{\ol{i+1}} {\genK}_{i}^{-1} ) , \qquad
 {\genE}_{\ol{i}} {\genF}_{j} - {\genF}_{j} {\genE}_{\ol{i}}
	= \delta_{i,j}  ( {\genK}_{i+1} {\genK}_{\ol{i}}  - {\genK}_{\ol{i+1}} {\genK}_{i} ) ;\\
({\rm QQ5}) \quad
&{\genE}_{\ol{i}}^2 = -\frac{ {v} - {v}^{-1} }{{v} + {v}^{-1}} {\genE}_{i}^2, \quad
	{\genF}_{\ol{i}}^2 = \frac{ {v} - {v}^{-1} }{{v} + {v}^{-1}} {\genF}_{i}^2, \\
&
{\genE}_{i} {\genE}_{\ol{j}} - {\genE}_{\ol{j}} {\genE}_{i}
	=  {\genF}_{i} {\genF}_{\ol{j}} - {\genF}_{\ol{j}} {\genF}_{i}
	= 0 ,  \quad \mbox{ for } |i - j| \ne 1,
\\
& 
{\genE}_{i} {\genE}_{j} - {\genE}_{j} {\genE}_{i} = {\genF}_{i} {\genF}_{j} - {\genF}_{j} {\genF}_{i}
= {\genE}_{\ol{i}}{\genE}_{\ol{j}}  + {\genE}_{\ol{j}}  {\genE}_{\ol{i}}= {\genF}_{\ol{i}} {\genF}_{\ol{j}}  + {\genF}_{\ol{j}} {\genF}_{\ol{i}}
	= 0 \quad \mbox{ for }\quad |i-j|  > 1, 	
\\
& {\genE}_{i} {\genE}_{i+1} - {v} {\genE}_{i+1} {\genE}_{i}
	= {\genE}_{\ol{i}} {\genE}_{\ol{i+1}} + {v} {\genE}_{\ol{i+1}} {\genE}_{\ol{i}}, \qquad
  {\genE}_{i} {\genE}_{\ol{i+1}} - {v} {\genE}_{\ol{i+1}} {\genE}_{i}
	= {\genE}_{\ol{i}} {\genE}_{i+1} - {v} {\genE}_{i+1} {\genE}_{\ol{i}}, \\
& {\genF}_{i} {\genF}_{i+1} - {v} {\genF}_{i+1} {\genF}_{i}
	= - ({\genF}_{\ol{i}} {\genF}_{\ol{i+1}} + {v} {\genF}_{\ol{i+1}} {\genF}_{\ol{i}}), \qquad
  {\genF}_{i} {\genF}_{\ol{i+1}}  - {v} {\genF}_{\ol{i+1}}  {\genF}_{i}
	=  {\genF}_{\ol{i}} {\genF}_{i+1} - {v} {\genF}_{i+1} {\genF}_{\ol{i}} ;\\
({\rm QQ6}) \quad
& {\genE}_{i}^2 {\genE}_{j} - ( {v} + {v}^{-1} ) {\genE}_{i} {\genE}_{j} {\genE}_{i} + {\genE}_{j}  {\genE}_{i}^2 = 0, \qquad
	{\genF}_{i}^2 {\genF}_{j} - ( {v} + {v}^{-1} ) {\genF}_{i} {\genF}_{j} {\genF}_{i} + {\genF}_{j}  {\genF}_{i}^2 = 0, \\
& {\genE}_{i}^2 {\genE}_{\ol{j}} - ( {v} + {v}^{-1} ) {\genE}_{i} {\genE}_{\ol{j}} {\genE}_{i} + {\genE}_{\ol{j}}  {\genE}_{i}^2 = 0, \qquad
	{\genF}_{i}^2 {\genF}_{\ol{j}} - ( {v} + {v}^{-1} ) {\genF}_{i} {\genF}_{\ol{j}} {\genF}_{i} + {\genF}_{\ol{j}}  {\genF}_{i}^2 = 0,  \\
& \qquad \mbox{ where } \quad |i-j| = 1.
\end{align*}
\end{defn}

\begin{rem}\label{induct_all}
By the generating relations of {$\Uvqn$}, we have
\begin{equation}\label{rel_ind_EFK}
\begin{aligned}
{\genE}_{\ol{j}} 
&= - {v} {\genK}_{j+1} {\genK}_{\ol{j+1}} {\genE}_{j} + {v}^{-1} {\genE}_{j} {\genK}_{j+1}  {\genK}_{\ol{j+1}} , \\ 
{\genF}_{\ol{j}} 
&= {v}  {\genK}_{j+1}^{-1} {\genK}_{\ol{j+1}} {\genF}_{j} - {v}^{-1} {\genF}_{j} {\genK}_{j+1}^{-1}  {\genK}_{\ol{j+1}},\\
{\genK}_{\ol{j}}
&=  {v} {\genE}_{j}  {\genK}_{j+1}^{-1} {\genK}_{\ol{j+1}} {\genF}_{j}    {\genK}_{j+1}  
	 - {v}^{-1} {\genE}_{j}  {\genF}_{j} {\genK}_{j+1}^{-1}  {\genK}_{\ol{j+1}}   {\genK}_{j+1}  
		- {v}  {\genK}_{j+1}^{-1} {\genK}_{\ol{j+1}} {\genF}_{j}  {\genE}_{j} {\genK}_{j+1}  \\
		& \qquad 
		+ {v}^{-1} {\genF}_{j} {\genK}_{j+1}^{-1}  {\genK}_{\ol{j+1}} {\genE}_{j} {\genK}_{j+1} 
		+  {\genK}_{j}^{-1} {\genK}_{\ol{j+1}} {\genK}_{j+1} ,
\end{aligned}
\end{equation}
which means  {$\genE_{\ol{j}}$},  {$\genF_{\ol{j}}$},  {$\genK_{\ol{j}}$} 
for all {$1\le j \le n-1$} could be obtained from  the even generators 
and  {$\genK_{\ol{n}}$}.
\end{rem}

\subsection{The twisted queer $q$-Schur superalgebra }
\

The Clifford algebra {$\CL_r$} is an algebra over {$\QQ$} generated by {$c_1, \cdots , c_r$} subject to the relations 
\begin{align}\label{cliff}
	c_i^2 = -1,  \quad c_i c_j = - c_j c_i , \quad \mbox{ for } 1 \le i \ne j \le r.
\end{align}
Let {$\fS_{r}$} be the symmetric group on {$r$} letters, 
with generators  {$ s_{i}=(i, i+1) $} for all {$1 \le i < r$}. 
The Hecke  algebra  {$\Heck$} associated with {$\fS_{r}$}  
is the algebra  over {$\Qv$}  generated by 
{$T_{i} = T_{s_{i}}$} for all  {$1 \le i < r$},
subject to the relations
\begin{align}
	&(T_i - {q} )(T_i + 1) = 0, \qquad T_i T_j = T_j T_i,
	\quad \mbox{ for } 1 \le i, j \le r-1, \quad |i-j| >1;\label{Hecke-1} \\
	&T_i T_{i+1} T_i = T_{i+1} T_i T_{i+1}, \quad \mbox{ for } 1 \le i \le r-2.\label{Hecke-2}
\end{align}
If {$w = s_{i_1} \cdots s_{i_k} \in \fS_r$} is a reduced expression, set {$T_w = T_{i_1} \cdots T_{i_k}$}.
For any  {$J \subseteq \fS_{r}$}, denote 
\begin{align*}
	x_{J} = \sum_{w \in J} T_w.
\end{align*}
Let {$\CMN(n,r)=\{\lambda=(\lambda_1,\lambda_2,\ldots,\lambda_n) \where\lambda_i\in\NN, \ \sum_i\lambda_i=r\} \subset \NN^n$} 
be the set of  compositions of $r$ with $n$ parts.
For any {$\lambda \in \CMN(n,r)$}, 
denote by {$\fS_{\lambda}$} the standard Young subgroup of {$\fS_r$} corresponding to {$\lambda$}, 
and denote
\begin{align}\label{xlambda}
	x_{\lambda} = x_{\fS_{\lambda}} = \sum_{w \in \fS_{\lambda}} T_w.
\end{align}
By \cite[Lemma 7.32]{DDPW}, for any {$T_{i}$} satisfying {$  s_i \in \fS_{\lambda} $}, we have 
\begin{align}\label{Tixlambda}
	T_i x_{\lambda} =x_{\lambda}  T_i = {q} x_{\lambda}.
\end{align} 
The Hecke-Clifford superalgebra  {$\HCR$}  is a  superalgebra  over {$\Qv$}  generated by {$\Heck$} and {$\CL_r$}, 
with  {$T_1, \cdots, T_{r-1}$} as  even  generators and {$c_1, \cdots, c_r$} as odd generators, 
subject to the relations \eqref{cliff}-\eqref{Hecke-2} and the extra relations
\begin{align}
	&T_i c_j = c_j T_i,\qquad \text{ for } 1\leq i\leq r-1, 1\leq j\leq r, j\neq i, i+1\label{Hecke-Cliff1} \\
&T_i c_i = c_{i+1} T_i, \qquad T_i c_{i+1} = c_i T_i - ({q}-1)(c_i - c_{i+1}),\qquad 1 \le i  \le r-1.\label{Hecke-Cliff2}
\end{align}

Following \cite[Section 4]{DW2018}, we introduce the elements in $\HCR$ for $1\leq i\leq j\leq r$:
\begin{align}\label{cqij}
	&c_{{q},i,j} = {q}^{j-i}c_i +  {q}^{j-i-1}c_{i+1} + \cdots + {q} c_{j-1} + c_j,\quad
	&c'_{{q},i,j} = c_i + {q} c_{i+1} + \cdots + {q}^{j-i}c_j.
\end{align}
For  {$\lambda, \alpha \in \NN^n $}, let {$\tilde{\lambda}_k = \sum_{i=1}^{k}  {\lambda}_i$}, and recall   the following elements introduced in \cite[Section 4]{DW2018}: 
\begin{equation}\label{eq_c_a}
\begin{aligned}
	&c^{\alpha}_{\lambda} = {(c_{{q}, 1, \tilde{\lambda}_1})}^{\alpha_1}
				{(c_{ {q}, \tilde{\lambda}_1+1 , \tilde{\lambda}_2} )}^{\alpha_2}
				\cdots
				{(c_{{q}, \tilde{\lambda}_{N-1}+1,  \tilde{\lambda}_n} )}^{\alpha_n}, \\
	&(c^{\alpha}_{\lambda})' = {(c'_{{q}, 1, \tilde{\lambda}_1})}^{\alpha_1}
				{(c'_{{q}, \tilde{\lambda}_1 + 1,  \tilde{\lambda}_2} )}^{\alpha_2}
				\cdots
				{(c'_{{q}, \tilde{\lambda}_{N-1} + 1,  \tilde{\lambda}_n })}^{\alpha_n} .
\end{aligned}
\end{equation}
By \cite[Lemma 4.1 and Corollary 4.3]{DW2018}
we have
\begin{align}\label{xc}
	x_{\lambda} c_{\lambda}^{\alpha} = (c_{\lambda}^{\alpha})'  x_{\lambda}.
\end{align}

Recall the following matrix sets introduced in \cite{DW2015}:
\begin{equation}\label{labelsets}
\aligned
&\MNR(n,r) =\{A=(a_{i,j})\in \MN(n)\mid r=|A|\},\\
	& \MNZN(n) = \{ \Ad=(\SUP{A}) \where \SE{A} \in \MN(n), \SO{A} \in M_n(\ZG) \}, \\
	& \MNZ(n,r) = \{ \Ad=(\SUP{A}) \in \MNZN(n) \where  \SE{A} + \SO{A} \in \MNR(n,r) \}, \\
	& \MNZNS(n,r) = \{ \Ad=(\SUP{A}) \in \MNZN(n) \where  \SE{A}=(\SEE{a}_{i,j}),\  \SEE{a}_{i,i}=0 \mbox{ for all } 1 \le i \le i \}.
	\endaligned
\end{equation}
Let $O$ be the zero matrix in $\MN(n)$ and $M_n(\ZG) $,
and we denote {$\bs{O}=(O | O) $} be the zero matrix in {$\MNZN(n) $}.

Suppose $\Ad=(\SE{A}|\SO{A})\in\MNZ(n,r)$ 
with {$\SE{A}=(\SEE{a}_{i,j})$}, {$\SO{A}=(\SOE{a}_{i,j})$}, 
let $A = \SE{A} + \SO{A} = ({a}_{i,j})$, 
$ \snorm{A}  = \sum (\SEE{a}_{i,j}  + \sum \SOE{a}_{i,j}) $
.
The parity of {$\Ad$} is defined to be {$\parity{\Ad} = \sum \SOE{a}_{i,j} \in \ZG$}.
Recall the elements
\begin{align*}
&\nu_A=(a_{1,1},\ldots,a_{n,1},\ldots,a_{1,n},\ldots,a_{n,n}), \\
&\nu_{\SO{A}}=(a_{1,1}^1,\ldots, a_{n,1}^1,\ldots,a^1_{1,n},\ldots,a^1_{n,n})
\end{align*}
and set $\lambda=\ro(A)$. Recall  the following important elements in $\HCR$ introduced in \cite[Section 5, (5.0.3), (5.0.4)]{DW2018} :
\begin{equation}\label{eq cA}
	c_{\Ad} =  c_{\nu_{A}}^{\nu_{\SO{A}} }  , \qquad
	T_{\Ad} =  x_{\lambda} T_{d_{A}} c_{\Ad}
			\sum _{{\sigma} \in \D_{{\nu}_{A}} \cap \fS_{\mu}} T_{\sigma}.
\end{equation}

Referring to \cite[Section 2]{DGLW2025}, 
for any {$w \in \HCR$} 
and {${\Ad} \in \MNZ(n, r)$},
there is a supermodule automorphism defined as 
\begin{equation}\label{def_Phi}
\begin{aligned}
 {\Phi}_{\Ad} : \bigoplus_{\mu \in \CMN(n, r)} x_{\mu}\HCR & \longrightarrow \bigoplus_{\mu \in \CMN(n, r)} x_{\mu}\HCR,  \\
		x_{\mu} w & \mapsto {(-1)}^{\parity{{\Ad}} \cdot \parity{ w}} {\phi}_{\Ad}(x_{\mu} w) 
						=  {(-1)}^{\parity{{\Ad}}  \cdot \parity{ w}}   \delta_{\mu, \co(A)}  T_{\Ad} \cdot  w.
\end{aligned}
\end{equation}
For any supermodule $M$ over superalgebra $K$, 
we  denote the endomorphism superalgebra of $M$ by  $ \End^s_{K}(M)$.
To realize the superalgebra $\Uvqn$, 
the authors in \cite{DGLW2025} constructed 
the twisted  queer ${q}$-Schur superalgebra  $\SQvnrR$ as :
\begin{align*}
\SQvnrR 
= \End^s_{\HCR}(\bigoplus_{\lambda \in \CMN(n,r)} x_{\lambda} \HCR),
\end{align*}
and $\SQvnrR$ has a basis
  {$ \{ \Phi_{\Ad} \where  \Ad \in \MNZ(n, r)\}$},
  see  \cite{DGLW2025} for more details.

\subsection{The realizations of {$\Uvqn$}}
\ 

	For any {$\Ad=(\SUP{A}) \in \MNZNS(n)$}, {$\bs{j} \in {\ZZ}^{n} $},
recall 
$	\Norm(\Ad):=\sum_{i\geq k,j<l}a_{i,j}a_{k,l}+\sum_{i,j}a^{\bar 1}_{i,j}a^{\bar0}_{i,j}$ 
and the standard element  $[\Ad]={v}^{-\Norm(\Ad)}\Phi_\Ad $ in \cite[Definition 3.4 and Proposition 3.5]{DGLW2025}.
The long elements in {$\SQvnrR$} is defined in \cite[5.0.2]{DGLW2025}:
\begin{equation}\label{def_ajr}
\begin{aligned}
	\AJRS(\Ad, \bs{j}, r) =
\left\{
\begin{aligned}
	& \sum_{\substack{\lambda \in \CMN(n, r-\snorm{A})} }
	 {v}^{\lambda\centerdot \bs{j}} [ \SE{A} + \lambda | \SO{A} ],
		\ &\mbox{if } \snorm{A} \le r; \\
	&0, &\mbox{otherwise},
\end{aligned}
\right.
\end{aligned}
\end{equation}
where $\lambda \centerdot \bs{j}=\sum_{i=1}^n\lambda_ij_i$ is the dot product.

Recall the notations in \cite[Section 8]{DGLW2025},
\begin{equation}\label{def_aj}
\begin{aligned}
    &\AJS(\Ad, \bs{j}) = \sum_{r \ge 1 } \AJRS(\Ad, \bs{j}, r) \in \prod_{r \ge 1 } \SQvnrR,
        \quad \mbox{for any } \Ad \in \MNZNS(n), \bs{j} \in {\ZZ}^{n}, \\
    &\USnv = \tspan_{\Qv} \{\AJS(\Ad, \bs{j})\  \where  \Ad \in \MNZNS(n), \bs{j} \in {\ZZ}^{n} \} \subset \prod_{r \ge 1 } \SQvnrR.
\end{aligned}
\end{equation}
It is trivial that there is a  superalgebra epimorphism
\begin{align*}
	\bs{\eta}_r : \USnv & \longrightarrow  \SQvnrR, \\
	\AJS(\Ad, \bs{j}) & \mapsto \AJR(\Ad, \bs{j}, r).
\end{align*}
\begin{thm}\label{map_realization}
(See \cite{DGLW2025})
There is a {$\Qv$}-superalgebra isomorphism {$\bs{\xi}_n: \Uvq(n) \to \USnv$}, mapping
\begin{align*}
&{\genK}_{i}^{\pm 1} \to \AJS(\bs{O}, \pm \bs{\ep}_{i}), \qquad
{\genK}_{\ol{i}} \to  \ABJS(O, E_{i,i}, \bs{0}), \\
&{\genE}_{j} 	\to \ABJS(E_{j, j+1}, O, \bs{0}) , \qquad
{\genE}_{\ol{j}} \to   \ABJS( O, E_{j,j+1}, \bs{0}),  \\
&{\genF}_{j}	\to \ABJS(E_{j+1, j}, O, \bs{0}) , \qquad
{\genF}_{\ol{j}} \to  \ABJS(O, E_{j+1,j},  \bs{0}),
\end{align*}
where {$1 \le i \le n, 1 \le j \le n-1$}.
\end{thm}

Denote 
\begin{align*}
&\AK(h,k) = \sum_{u=1}^{k-1} a_{h, u} ,\qquad 
\BK(h,k) = \sum_{j=k+1}^{n} a_{h, j} ,  \\
&[\![k ]\!]_{x} = 1 + {x} + {x}^2 + \cdots + {x}^{k-1},\qquad
[\![k ]\!]_{x, y}  =  [\![k ]\!]_{x} -  [\![k ]\!]_{y} .
\end{align*}
Referring to Theorem \ref{map_realization},
and applying  \cite[Section 1.4, Main Theorem]{DGLW2025},
we have 
\begin{thm}\label{mulformAonPhi}
For any  {${\Phi}_{\Ad} \in \SQvnrR$} and {$\ro(A) = \lambda$},
we have the following formulas:
\begin{align*}
{\rm (1)} \quad
{\genK}_{h} \Phi_{\Ad} 
&= 
	{v}^{\lambda_h  }	{\Phi}_{\Ad} , \\
{\rm (2)} \quad
{\genE}_{h} {\Phi}_{\Ad} 
& =  {v}^{-\lambda_h}  \sum_{k=1}^n 
	{v}^{ 2 \BK(h,k) + 2 \SOE{a}_{h+1,k}}  \STEP{ \SEE{a}_{h,k} + 1}  
		{\Phi}_{(\SE{A} - E_{h+1, k} + E_{h,k} | \SO{A} )}  \\
	&\qquad +
	  {v}^{-\lambda_h}  \sum_{k=1}^n 
	{v}^{ 2 \BK(h,k)} {\Phi}_{(\SE{A}  | \SO{A} - E_{h+1, k}  + E_{h,k} )}  \\
	&\qquad +
	  {v}^{-\lambda_h}  \sum_{k=1}^n 
	{v}^{2 \BK(h,k) - 2}  \STEPPD{a_{h,k}+1}
	{\Phi}_{(\SE{A} + 2E_{h,k} | \SO{A}  -E_{h,k} - E_{h+1,k} )},  \\
{\rm (3)} \quad
{\genF}_{h} {\Phi}_{\Ad} 
&= {v}^{- \lambda_{h+1}}\sum_{k=1}^n 
	 {v }^{ 2 \AK(h+1,k) }  \STEP{ \SEE{a}_{h+1, k} +1} 
	{\Phi}_{(\SE{A} - E_{h,k} + E_{h+1, k} | \SO{A} )} \\
	& \qquad + 
	 {v}^{- \lambda_{h+1}}\sum_{k=1}^n
	{v}^{2 \AK(h+1,k)  + 2 a_{h, k} - 2} {\Phi}_{(\SE{A}  |\SO{A} - E_{h,k} + E_{h+1,k})} \\
	& \qquad -
	 {v}^{- \lambda_{h+1}}\sum_{k=1}^n
	{v}^{2 \AK(h+1,k)  + 2 a_{h, k} - 4} 
	{\STEPPDR{  a_{h+1, k} +1}}
	{\Phi}_{(\SE{A} + 2E_{h+1,k} | \SO{A}  - E_{h,k} - E_{h+1,k})}, \\
{\rm (4)} \quad
{\genK}_{\ol{n}} \Phi_{\Ad} 
&= 
	{v}^{ - \lambda_n + 1} 
\sum_{ k=1}^n 
	{(-1)}^{{\SOE{\tilde{a}}}_{n-1,k}   + \parity{\Ad}} 
	{v}^{ 2 \BK(n,k) } {\Phi}_{(\SE{A} -E_{n,k}|\SO{A}+ E_{n,k})}  \\
	&\qquad + 
		{v}^{ - \lambda_n + 1} 
		\sum_{ k=1}^n 
	{(-1)}^{ {\SOE{\tilde{a}}}_{n-1,k}   + \parity{\Ad} + 1} 
	{v}^{ 2 \BK(n,k) }  {\STEPP{a_{n,k}}} 
	{\Phi}_{(\SE{A} +  E_{n,k}| \SO{A} - E_{n,k})}.
\end{align*}
\end{thm}

\spaceintv
\section{The  weight modules  of {$\Uvqn$}  }\label{sec_modules}

In the following of this paper, we  denote $\Uvqn$ by  $\Uv$  for simplicity.
Let {$\Uv^+$} (resp., {$\Uv^-$}) be the subalgebra of {$\Uv$}
generated by {${\genE}_{j}, {\genE}_{\ol{j}}$} (resp., {${\genF}_{j}, {\genF}_{\ol{j}}$}) for all {$j=1, \cdots , n-1$}.
Let {$\Uv^0$} be the subalgebra of {$\Uv$}
generated by {${\genK}_{i}, {\genK}_{i}^{-1}, {\genK}_{\ol{i}}$}  for all {$i=1, \cdots , n$}.
Let {$\Uv^{\ge 0}$} be the subalgebra of {$\Uv$} generated by {$\Uv^+$} and {$\Uv^0$}.

Let {$P = {\Qv}  \setminus \{0\} $}.
Denote {$ {v}^{\ZZ} = \{ {v}^{i} \where i \in \ZZ\} $},
{$ {v}^{\NN} = \{ {v}^{i} \where i \in \NN\} $}.
For any  {$\Uv$}-module $M$ and  {$\omega = (\omega_1, \cdots, \omega_n) \in P^n$}, 
put
\begin{align*}
M_{\omega} 
= \{ m \in M \where
	{\genK}_{i} m = \omega_{i}  m , 
	\mbox{ for all } 1 \le i \le n
	\}.
\end{align*}
If {$M_{\omega} \ne 0$}, {$\omega$} is called a weight of {$M$} and {$M_{\omega}$} is the weight space of weight {$\omega$}.
Denote the set of weights of $M$ by
\begin{align*}
\wt(M) = \{ \omega \in P^n \where M_{\omega} \ne 0\}.
\end{align*}
A  {$\Uv$}-module $M$  is called a weight module if 
\begin{align*}
M = 
	\bigoplus_{\omega \in \wt(M) } M_{\omega} .
\end{align*}
A nonzero vector {$m \in M_{\omega} $}  is called a highest weight vector of weight {$\omega$}  
if {${\genE}_{j} m = 0$}, {${\genE}_{\ol{j}} m = 0$}  for all {$1 \le j \le n-1$}.
A  {$\Uv$}-module $M$ is called a highest weight module of highest weight {$\omega$} if 
it is generated by a highest weight vector of weight {$\omega$}.

For any weight module $M$  and {$m \in M$}, 
it is clear $m$ has a unique decomposition {$m  = \sum_{\omega \in \wt(M)} m_{\omega}$},
where {$ m_{\omega} \in M_{\omega}$} for each {$\omega \in \wt(M)$}.
Then we denote {$\wt(m) = \{ \omega \in \wt(M) \where m_{\omega} \ne 0\}$}.
A weight module $M$ is called a polynomial weight module if {$\wt(M) \subset {({v}^{\NN})}^n$}.

For any {$\omega \in P^n$}, 
let {$I(\omega)$} be the left ideal of  {$\Uv$} 
generated by {$\{ {\genK}_{i} - {\omega_i} \cdot 1, {\genE}_{j}, {\genE}_{\ol{j}} \where  1 \le i \le n,  1 \le j \le n-1 \}$}. 
We define the Verma module corresponding to weight {$\omega$} to be  {$V(\omega) = \Uv / I(\omega)$}.
Similar to \cite[Section 10.1]{CP}, 
we have the following
\begin{prop}
Any highest weight module of highest weight {$\omega$} is isomorphic to a quotien of {$V(\omega)$}. 
\end{prop}
\begin{proof}
Assume $M$ is a highest weight module of highest weight {$\omega$} generated by weight vector {$m$},
more precisely, {$M = \Uv m$}.
Then there is an epimorphism of modules {$f: V(\omega) \to M$}, 
mapping {$(1+ I(\omega))$} to {$m$},
and {$M \cong V(\omega)/{\mathrm{Ker}f}$}.
\end{proof}

For any {$\bs{k} \in \ZZ^n$} and {$\omega \in P^n$}, 
we denote
\begin{align*}
& {v}^{\bs{k}} =({v}^{k_1} , \cdots,  {v}^{k_n} ) \in P^n  , \\
& {v}^{\bs{k}} \omega= \omega	{v}^{\bs{k}} =({v}^{k_1} \omega_1, \cdots,  {v}^{k_n} \omega_n) \in P^n  , \\
& {\omega}^{-1} =( \omega_1^{-1}, \cdots,   \omega_n^{-1} ) \in P^n  .
\end{align*}

\begin{prop}\label{wt_shift}
Let {$M_{\omega}$} be a weight  space of a {$\Uv$}-module {$M$}, then 
\begin{align*}
&{\genK}_{i}^{\pm} M_{\omega} \subset M_{   \omega  }, \qquad
{\genK}_{\ol{i}} M_{\omega} \subset M_{ \omega }, \\
&{\genE}_{j} M_{\omega} \subset M_{{v}^{{\alpha}_j} \omega }, \qquad
{\genE}_{\ol{j}} M_{\omega} \subset M_{{v}^{{\alpha}_j}  \omega }, \\
&{\genF}_{j} M_{\omega} \subset M_{{v}^{-{\alpha}_j}  \omega }, \qquad
{\genF}_{\ol{j}} M_{\omega} \subset M_{{v}^{-{\alpha}_j}  \omega} ,
\end{align*}
for all {$ 1 \le i \le n$},  {$ 1 \le j \le n-1$}.
Furthurmore,
 {$M_{\omega}$} is a {$\Uv^0$}-module.
\end{prop}
\begin{proof}
The first relation is trivial.
Choose any non-zero {$m \in M_{\omega}$},
then for any {$ 1 \le i, j \le n$}, 
by the definition of weight spaces, we have 
\begin{align*}
&{\genK}_{i}   {\genK}_{\ol{j}}  m  =   {\genK}_{\ol{j}} {\genK}_{i}   m  
=   {\omega}_{i} {\genK}_{\ol{j}} m  ,
\end{align*}
hence we have {$ {\genK}_{\ol{j}} m \in M_{\omega}$}
and  {$M_{\omega}$} is a {$\Uv^0$}-module.

For any {$ 1 \le i \le n$}, {$ 1 \le  j \le n-1$}, 
referring to the generating relations of {$\Uv$}, we have 
\begin{align*}
{\genK}_{i} ({\genE}_{j} m) 
	&= {v}^{( {\ep}_i, \alpha_j)} {\genE}_{j}  {\genK}_{i} m 
	=  {v}^{( {\ep}_i, \alpha_j)}  {\omega}_{i}  {\genE}_{j}  m,\\
{\genK}_{i} ({\genF}_{j} m) 
	&= {v}^{-( {\ep}_i, \alpha_j)} {\genF}_{j}  {\genK}_{i} m 
	=  {v}^{-( {\ep}_i, \alpha_j) }  {\omega}_{i}  {\genF}_{j}  m,
\end{align*}
which means
\begin{align*}
{\genE}_{j} m \in M_{{v}^{{\alpha}_j}  \omega  }, \qquad
{\genF}_{j} m \in M_{{v}^{-{\alpha}_j}  \omega }.
\end{align*}
The proofs for {${\genE}_{\ol{j}}$} and {${\genF}_{\ol{j}}$} are similar to  {$E_{{j}}$}, {$F_{{j}}$}
and  we omit them.
\end{proof}

We define a partial order of the weights  as: 
for any two weights {$\omega, \gamma \in P^n$},
we say {$\omega \ge \gamma $} if there exists {$\bs{\beta} \in \bs{Q}^+$} such that {$ \omega  ={v}^{\bs{\beta}} { \gamma}$}.
If {$\omega \ge \gamma $} and {$\omega \ne \gamma $},
 denote {$\omega > \gamma $}.
A weight {$\omega \in \wt(M)$} is called maximal 
if there is no other {$  \gamma \in \wt(M)$} such that {$\gamma > \omega$}.

\begin{prop}\label{fd_weight}
Let {$M$} be 
a finite dimensional 
irreducible weight module of  {$\Uv$},
then there exists only one maximal  {$\omega \in \wt(M) $}
and {$M$} is a highest weight module of highest weigh  {$\omega$}.
\end{prop}
\begin{proof}
As {$M$} is a finite dimensional module, the set {$\wt(M)$} is finite.
Then there is at least one maximal {$\omega$}.
It is clear that  for any {$1 \le j \le  n-1$},
we have {${\genE}_{j} M_{\omega} =0 $},
otherwise by Proposition \ref{wt_shift},
{${v}^{ {\alpha}_j} \omega \in \wt(M)$} and {${v}^{ {\alpha}_j}  \omega > \omega   $}, 
which causes a contradiction.
If there is another maximal {${\lambda}\in \wt(M) $} and {${\lambda} \ne \omega$}, 
then  {$\Uv M_{\lambda}$} and {$\Uv  M_{\omega}$} are different submodules of {$M$},
which contradicts to the assumption that $M$ is irreducible.
Hence  {$\omega$}  is the only maximal weight, which is of course a highest weight,
 and there exists {$ m_{\omega} \in  M_{\omega}$} such that {$M = \Uv m_{\omega}$}.
\end{proof}

\begin{prop}\label{fd_weight_eig}
Let {$M$} be a finite dimensional 
irreducible weight module of  {$\Uv$},
then there is no non-zero eigenvalues for {${\genE}_{j}$}, {${\genF}_{j}$}, {${\genE}_{\ol{j}}$}, {${\genF}_{\ol{j}}$},
for all 
{$1 \le j \le  n-1$}. 
\end{prop}
\begin{proof}
Assume there exist   $ 1 \le j \le n-1$,   {$m \in M$} and nonzero {$ {\cft} \in \Qv$} 
satisfying  {${\genE}_{j} m =  {\cft}  m$}.
By  {$m=\sum_{\lambda \in \wt(m)}  m_{\lambda} $},
we have {${\genE}_{j} m  =\sum_{\lambda \in \wt(m)}{\genE}_{j}   m_{\lambda} $}.
Hence  Proposition \ref{wt_shift} implies
\begin{equation}\label{wt_shift_ej}
\wt({\genE}_{j} m)  = \{  {v}^{ {\alpha}_j}  \lambda  \where \lambda \in \wt(m)\}.
\end{equation}
On the other hand, 
because 
{${\genE}_{j} m  =\sum_{\lambda \in \wt(m)} {\cft}  m_{\lambda} $},
we have {$\wt({\genE}_{j} m)  =  \wt(m)$},
which contradicts to \eqref{wt_shift_ej}.
Hence there is no nonero eigenvector for {${\genE}_{j}$}.
The proof for {${\genF}_{j}$}, {${\genE}_{\ol{j}}$}, {${\genF}_{\ol{j}}$} is analogous and we omit it.
\end{proof}

\begin{prop}\label{fd_eigen}
Let {$M$} be a polynomial weight module of  {$\Uv$} and $1 \le i \le n$.
If {${\genK}_{\ol{i}}$} has an eigenvector  {$m$} corresponding to eigenvalue {$ {\cft} \in \Qv$},
then 
{$ {\cft}  =\pm 1$} and {$m \in M_{\lambda}$} with {$\lambda_i =\pm {v}$}.
\end{prop}

\begin{proof}
Assume {${\genK}_{\ol{i}}$} has an eigenvector {$m$} corresponding to eigenvalue {$t \in \Qv$},
{${\genK}_{\ol{i}} m ={\cft} m$}.
Denote  {$m=\sum_{\lambda \in \wt(m)} m_{\lambda} $}.
For any {$\lambda = (v^{k_1}, \cdots, v^{k_n})\in \wt(m)$}, 
Lemma \ref{wt_shift}  implies {${\genK}_{\ol{i}}  m_{\lambda} \in M_{\lambda}$}, 
hence we have  {${\genK}_{\ol{i}}  m_{\lambda} ={\cft}  m_{\lambda}$} and  {${\genK}_{\ol{i}}^2  m_{\lambda} ={\cft}^2  m_{\lambda}$}. 
On the other hand,  by the last relation of (QQ1), 
we have 
\begin{align*}
{\genK}_{\ol{i}}^2  m_{\lambda} 
&=  \frac{{\genK}_{i}^2 - {\genK}_{i}^{-2}}{{v}^2 - {v}^{-2}} m_{\lambda} \\
&=  \frac{ {v}^{2 k_i} -   {v}^{- 2 k_i} } 
				{{v}^2 - {v}^{-2}} m_{\lambda} \\
&=   ({v}^{2(k_i-1)} + {v}^{2(k_i-3)} + \cdots + {v}^{-2(k_i-1)} ) m_{\lambda},
\end{align*}
which means
\begin{align*}
 {\cft} ^2 = {v}^{2(k_i-1)} + {v}^{2(k_i-3)} + \cdots + {v}^{-2(k_i-1)}  
\end{align*}
and 
\begin{align*}
 {\cft}   =  \pm  \sqrt{  {v}^{2(k_i-1)} + {v}^{2(k_i-3)} + \cdots + {v}^{-2(k_i-1)}   }.
\end{align*}
By {$ {\cft} \in \Qv$}, we have   {$  \sqrt{  {v}^{2(k_i-1)} + {v}^{2(k_i-3)} + \cdots + {v}^{-2(k_i-1)}   } \in \Qv$}, 
{$2(k_i-1) = - 2(k_i-1)$}.
Consequencely, {$k_i =1$},  and {$ {\cft}  = \pm 1$}.
By 
\begin{align*}
   \frac{{\lambda}_{i}^2 -  {\lambda}_{i}^{-2}} {{v}^2 - {v}^{-2}}  
   = {\cft}^2  
   = 1
\end{align*}
and {$\lambda_i \in \Qv$},
we conclude {$\lambda_i = \pm {v}$}.
\end{proof}

For any {$t \in \NN  $}, recall the notations 
\begin{align*}
&\SSTEP{t} = \frac{ {v}^{t} - {v}^{-t} }{ {v} - {v}^{-1} } =  {v}^{t-1}+ {v}^{t-3}+ \cdots + {v}^{-t+3}+ {v}^{-t+1} , \\
&\SSTEP{t} ! = \SSTEP{t} \SSTEP{t-1} \cdots \SSTEP{1} ,
\end{align*}
and denote {$\SSTEP{0} = 1$}, {$\SSTEP{0} ! = 1$}.
Recall the divided power  for any X as 
\begin{align*}
X_{j}^{(t)} = \frac{1}{\SSTEP{t} !} X_{j}^{t}.
\end{align*}

\begin{prop}\label{mp_shift}
Let {$m$} be a highest weight vector of weight {$\omega$},
{$p \in \NN$}.
Set {$m_{j,p} = {\genF}_{j}^{(p)} m$},
{$m_{j, 0} = m$} for {$1 \le j \le n$}, 
and denote {$ (\bs{\ep}_i, \alpha_{0}) = 0$} for $1 \le i \le n$.
Then 
\begin{align*}
&{\genK}_{i}  m_{j,p} = {\omega}_i  {v}^{-p (\bs{\ep}_i, \alpha_{j}) }   m_{j, p}  , \quad
{\genK}_{i}^{-1}  m_{j,p} =  {\omega}_i^{-1} {v}^{  p \cdot  (\bs{\ep}_i, \alpha_{j})}  m_{j, p}, \\
&{\genE}_{j} m_{j,p} 
=
		   \frac{   
		  	 \omega_{j} \omega_{j+1}^{-1}   {v}^{1-p} 
		  	-   \omega_{j}^{-1} \omega_{j+1}  {v}^{p-1} 
		  	  }  {  {v} - {v}^{-1} }
		  	  m_{j,p-1}  , \qquad
{\genF}_{j} m_{j,p} =  \SSTEP{p+1} m_{j, p+1}.\\
\end{align*}
\end{prop}
\begin{proof}
The result could be proved by direct calculations. 
First of all, 
referring to 
{$  {\genK}_{i}   {\genF}_{j}^{p} 
	 =  {v}^{  -p(\bs{\ep}_i, \alpha_{j})  }  {\genF}_{j}^{p} {\genK}_{i}  $} and  
{$  {\genK}_{i}^{-1}   {\genF}_{j}^{p} 
	 =  {v}^{  p(\bs{\ep}_i, \alpha_{j})  }  {\genF}_{j}^{p} {\genK}_{i}^{-1}  $}, 
we have 
\begin{align*}
{\genK}_{i} m_{j,0} 
	 &
	 =  {\omega}_i   m_{j, 0} , \\
{\genK}_{i} m_{j,p} 
	 &= \frac{1}{\SSTEP{p} !}  {\genK}_{i}   {\genF}_{j}^{p} m  
	 = \frac{1}{\SSTEP{p} !}  {v}^{  -p(\bs{\ep}_i, \alpha_{j})  }  {\genF}_{j}^{p} {\genK}_{i}   m  
	 =  {\omega}_i  {v}^{-p (\bs{\ep}_i, \alpha_{j}) }   m_{j, p} , \\
{\genK}_{i}^{-1} m_{j,p} 
	 &= \frac{1}{\SSTEP{p} !}  {\genK}_{i}^{-1}    {\genF}_{j}^{p} m  
	 = \frac{1}{\SSTEP{p} !}  {v}^{ p(\bs{\ep}_i, \alpha_{j})  }  {\genF}_{j}^{p} {\genK}_{i}^{-1}   m  
	 =  {\omega}_{i}^{-1}   {v}^{ p (\bs{\ep}_i, \alpha_{j}) }   m_{j, p} , \\
{\genF}_{j} m_{j,p} 
	&= {\genF}_{j} {\genF}_{j}^{(p)} m  
	 = \frac{1}{\SSTEP{p} !}  {\genF}_{j} \cdot {\genF}_{j}^{p} m 
	=  \SSTEP{p+1}  \cdot  \frac{1}{\SSTEP{p+1} !} {\genF}_{j}^{p+1} m 
	=  \SSTEP{p+1} m_{j, p+1}.
\end{align*}
For the second equation,
 considering  
\begin{align*}
& {\genE}_{j} m = 0,  \qquad
 {\genK}_{i}^{-1} m_{j,p} 
=   {\omega}_i^{-1} {v}^{  p \cdot  (\bs{\ep}_i, \alpha_{j})}  m_{j, p}, \\
&\SSTEP{t} 
=  {v}^{t-1}( 1 + {v}^{-2}+ \cdots + {v}^{-2t+2}  )
= {v}^{-t+1}( 1 + {v}^{2}+ \cdots + {v}^{2t-2}  ) 
\end{align*} 
and the first equation in \rm{(QQ4)},
we have
\begin{align*}
{\genE}_{j} m_{j,1} 
	&= {\genE}_{j} {\genF}_{j} m 
	= \frac{ {\genK}_{j} {\genK}_{j+1}^{-1} - {\genK}_{j}^{-1}{\genK}_{j+1}  }{{v} - {v}^{-1}} m + {\genF}_{j} {\genE}_{j} m \\
	&= \frac{ {\omega}_{j}   {\omega}_{j+1}^{-1}  - { \omega}_{j}^{-1}  {\omega}_{j+1} }{{v} - {v}^{-1}} m  ,\\
{\genE}_{j} m_{j,2} 
	&= \frac{1}{\SSTEP{2}}  {\genE}_{j}   {\genF}_{j} m_{j, 1} \\
	&= \frac{1}{\SSTEP{2}}  \frac{ {\genK}_{j} {\genK}_{j+1}^{-1} - {\genK}_{j}^{-1}{\genK}_{j+1} }{{v} -  {v}^{-1}} m_{j,1} 
		+ \frac{1}{\SSTEP{2}}  {\genF}_{j} {\genE}_{j} m_{j,1} \\
	&=
		\frac{1}{\SSTEP{2} ({v} -  {v}^{-1})}  ({\genK}_{j} {\genK}_{j+1}^{-1} - {\genK}_{j}^{-1}{\genK}_{j+1} ) m_{j,1} 
		+ \frac{1}{\SSTEP{2}}  {\genF}_{j} {\genE}_{j} m_{j,1} \\
	&=
		\frac{1}{\SSTEP{2} ({v} -  {v}^{-1})}  (  {\omega}_{j}  {v}^{- (\bs{\ep}_j, \alpha_{j}) }  {\omega}_{j+1}^{-1}  {v}^{ (\bs{\ep}_{j+1}, \alpha_{j}) }  
		-  {\omega}_{j}^{-1}  {v}^{ (\bs{\ep}_{j}, \alpha_{j}) }  {\omega}_{j+1}  {v}^{- (\bs{\ep}_{j+1}, \alpha_{j}) }   ) m_{j,1} \\
	& \qquad 
		+ \frac{1}{\SSTEP{2}}  {\genF}_{j} \frac{ {\omega}_{j}   {\omega}_{j+1}^{-1}  - { \omega}_{j}^{-1}  {\omega}_{j+1} }{{v} - {v}^{-1}} m \\
	&= 
		\frac{1}{\SSTEP{2} ({v} -  {v}^{-1})}  (  {\omega}_{j}  {v}^{-1  }  {\omega}_{j+1}^{-1}  {v}^{-1  }  
		-  {\omega}_{j}^{-1}  {v} {\omega}_{j+1}  {v}  ) m_{j,1} \\
	& \qquad 
		+ \frac{1}{\SSTEP{2} ({v} - {v}^{-1})}  ({\omega}_{j}   {\omega}_{j+1}^{-1}  - { \omega}_{j}^{-1}  {\omega}_{j+1} )  m_{j,1}  \\
	\\
&=   
		  \frac{   
		  	\omega_{j}  {\omega}_{j+1}^{-1}  {v}^{- 2}  -  \omega_{j}^{-1}    {\omega}_{j+1} {v}^{ 2}  + {\omega_{j} } {\omega}_{j+1}^{-1}  - { \omega_{j}^{-1}  {\omega}_{j+1}  }
		  }  { \SSTEP{2} ({v} - {v}^{-1})}
		  m_{j,1}   \\
&=   
		  \frac{   
		  	\omega_{j} {\omega}_{j+1}^{-1}   {v}^{- 1} -  \omega_{j}^{-1} {\omega}_{j+1}  {v}
		  }  { ({v} - {v}^{-1})}
		  m_{j,1} \\
\end{align*}
Consequently,
\begin{align*}
{\genE}_{j} m_{j,p} 
&= {\genE}_{j} {\genF}_{j}^{(p)} m \\
&= {\genE}_{j} \cdot  \frac{1}{\SSTEP{p}}  {\genF}_{j} m_{j,p-1} \\
&= \frac{1}{\SSTEP{p}}  {\genE}_{j}   {\genF}_{j} m_{j,p-1} \\
&= \frac{1}{\SSTEP{p}} 
		 \frac{ {\genK}_{j} {\genK}_{j+1}^{-1} - {\genK}_{j}^{-1}{\genK}_{j+1} }{{v} -  {v}^{-1}}
		  m_{j,p-1} 
		+ \frac{1}{\SSTEP{p}}  {\genF}_{j} {\genE}_{j} m_{j,p-1} \\
&= \frac{1}{\SSTEP{p}}  \frac{  \omega_{j}  {\omega}_{j+1}^{-1} {v}^{   - 2 p+2}  
			-  \omega_{j}^{-1}  {\omega}_{j+1} {v}^{  2p- 2}  }
		{{v} -  {v}^{-1}} m_{j,p-1}  \\
		& \qquad 
		+ \frac{1}{\SSTEP{p}}  {\genF}_{j} \cdot
		   \frac{   
		  	\omega_{j}  {\omega}_{j+1}^{-1}  (1+ {v}^{- 2} + \cdots + {v}^{-2(p-2)})  -  \omega_{j}^{-1}  {\omega}_{j+1} (1+  {v}^{ 2} + \cdots + {v}^{2(p-2)}) 
		  }  {  \SSTEP{p-1} ({v} - {v}^{-1})}
		  m_{j,p-2}   
		\\
&= \frac{1}{\SSTEP{p}}  \frac{  \omega_{j}  {\omega}_{j+1}^{-1} {v}^{   - 2 p+2}  
			-  \omega_{j}^{-1}  {\omega}_{j+1}  {v}^{  2p- 2}  }
		{{v} -  {v}^{-1}} m_{j,p-1} \\
		& \qquad 
		+ \frac{1}{\SSTEP{p}}   \cdot
		   \frac{   
		  	\omega_{j}  {\omega}_{j+1}^{-1} (1+ {v}^{- 2} + \cdots + {v}^{-2(p-2)})  -  \omega_{j}^{-1} {\omega}_{j+1}  (1+  {v}^{ 2} + \cdots + {v}^{2(p-2)}) 
		  }  {    ({v} - {v}^{-1})}
		  m_{j,p-1}  \\
&=
		   \frac{   
		  	\omega_{j}  {\omega}_{j+1}^{-1} (1+ {v}^{- 2} + \cdots + {v}^{-2(p-1)})  -  \omega_{j}^{-1}  {\omega}_{j+1}  (1+  {v}^{ 2} + \cdots + {v}^{2(p-1)}) 
		  }  {  \SSTEP{p}  ({v} - {v}^{-1})}
		  m_{j,p-1}  .
\end{align*}
Direct calculation shows
\begin{align*}
&
		   \frac{   
		  	 \omega_{j} \omega_{j+1}^{-1}  (1+ {v}^{- 2} + \cdots + {v}^{-2(p-1)})  -   \omega_{j}^{-1} \omega_{j+1} (1+  {v}^{ 2} + \cdots + {v}^{2(p-1)}) 
		  }  {  \SSTEP{p}  ({v} - {v}^{-1})}
\\
&=
		   \frac{   
		  	 \omega_{j} \omega_{j+1}^{-1}  \cdot  \frac{ 1 - {v}^{-2p} }{ {v}^{-2} (  {v}^{2} - 1 ) }   
		  	-   \omega_{j}^{-1} \omega_{j+1} \cdot  \frac{  {v}^{2p} - 1}{ {v}^{2} - 1} 
		  	  }  {  \frac{ {v}^{p} - {v}^{-p} }{ {v} - {v}^{-1} }  ({v} - {v}^{-1})}
 \\
&=
		   \frac{   
		  	 \omega_{j} \omega_{j+1}^{-1}  {v}^{2} ( 1 - {v}^{-2p} )
		  	-  \omega_{j}^{-1}(  {v}^{2p} - 1)
		  	  }  { ({v}^{2} - 1) ({v}^{p} - {v}^{-p}  )}
 \\ 
&=
		   \frac{   
		  	 \omega_{j} \omega_{j+1}^{-1}  {v}^{2}  {v}^{-p} ( {v}^{2p}  -1)
		  	-   \omega_{j}^{-1} \omega_{j+1}  {v}^{p}  (  {v}^{2p} - 1)
		  	  }  { ({v}^{2} - 1) ({v}^{2p} - 1)}
 \\ 
&=
		   \frac{   
		  	 \omega_{j} \omega_{j+1}^{-1}   {v}^{1-p} 
		  	-   \omega_{j}^{-1} \omega_{j+1}  {v}^{p-1} 
		  	  }  {  {v} - {v}^{-1} },
\end{align*}
then we have
\begin{align*}
{\genE}_{j} m_{j,p} 
&=
		   \frac{   
		  	 \omega_{j} \omega_{j+1}^{-1}  (1+ {v}^{- 2} + \cdots + {v}^{-2(p-1)})  -   \omega_{j}^{-1} \omega_{j+1} (1+  {v}^{ 2} + \cdots + {v}^{2(p-1)}) 
		  }  {  \SSTEP{p}  ({v} - {v}^{-1})}
		  m_{j,p-1}  \\
&=
		   \frac{   
		  	 \omega_{j} \omega_{j+1}^{-1}   {v}^{1-p} 
		  	-   \omega_{j}^{-1} \omega_{j+1}  {v}^{p-1} 
		  	  }  {  {v} - {v}^{-1} }
		  	  m_{j,p-1}  .
\end{align*}
\end{proof}

Recall the quantized universal enveloping algebra  {$U_v(\f{sl}_{n})$} from \cite[Section 0.4]{Lus3}.
\begin{defn} 
The algebra  {$U_v(\f{sl}_{n})$} is the associative algebra over  {$\Qv$} 
with generators {$E_{i}$},  {$F_{i}$},  {$K_{i}$},  {$1 \le i \le n-1$}, 
and generating relations 
\begin{align*}
&K_i K_j=K_jK_i,\qquad K_iK_i^{-1}=K_i^{-1} K_i=1;\\
&K_i E_j=v^{a_{i,j}}E_jK_i, \qquad K_i F_j=v^{-a_{i,j}}F_j K_i;\\
&E_i F_i-F_iE_i=\delta_{i,j}\frac{K_i-K_i^{-1}}{v-v^{-1}};\\
&E_i E_j=E_jE_i,\qquad F_i F_j=F_j F_i, \quad \mbox{if} \ |i-j| > 1;\\
&E_i^{2} E_j-(v + v^{-1})E_i E_j E_i+E_j E_i^{2}=0,\\
&F_i^{2} F_j-(v + v^{-1})F_i F_j F_i+F_j F_i^{2}=0, \quad \mbox{if}\  |i-j| = 1.
\end{align*}
\end{defn}
 For any {$1 \le j \le n-1$}, denote {$\wave{\genK}_{j} = {\genK}_{j} {\genK}_{j+1}^{-1}$}.
It is clear that there is an algebra homomorphism 
\begin{align*}
\pi_n: \quad 
&U_v(\f{sl}_{n}) \to \Uv, \\
&K_j \mapsto \wave{\genK}_j, \qquad
E_{j} \mapsto {\genE}_{j}, \qquad
F_{j} \mapsto {\genF}_{j}, \qquad
\mbox{ for } j=1, \cdots , n-1.
\end{align*}
The homomorphism  {$\pi_n$} induces a {$U_v(\f{sl}_{n}) $}-module structure on each $\Uv$-module.

By applying Proposition \ref{mp_shift} and with a similar discussion with  \cite[Theorem 10.1.7]{CP},
we have the following.
\begin{prop}
Every  finite dimensional irreducible {${\Uv}$}-module is a highest  weight module of highest weight {$\gamma \in ( \pm {v}^{\NN})^{n}$}. 
\end{prop}
\begin{proof}
The proof is analogous to \cite[Theorem 10.1.7]{CP}.

Assume   {$M$} is a finite dimensional irreducible {$\Uv$}-module.
Let {$\wQv$} be the algebraic closure of {$\Qv$},
and denote {$\wave{P} = \wQv \setminus \{0\} $} be the multiplication group of {$\wQv$}.
We denote 
\begin{align*}
\wave{M} = M {\otimes}_{\Qv} \wQv, \qquad
\wave{\Uv} =\Uv {\otimes}_{\Qv} \wQv.
\end{align*}
Then {$\wave{M}$} is a {$\wave{\Uv}$}-module.
Without confusion, for any {$X \in \Uv$},
the element {$X \otimes 1 \in \wave{\Uv}$} is still denoted as {$X $}.
View {$M$} as a subspace of {$\wave{M}$}, and {$\Uv$} as a subalgebra of {$\wave{\Uv}$}.

As {${\genK}_{i}$} commutes with {${\genK}_{j}$} for any {$i$} and {$j$}, 
there is a simultaneous eigenvector {$\wave{m} \in \wave{M} $} and {$\omega \in {\wave{P}}^{n}$}
such that {${\genK}_{i} \wave{m} = \omega_i \wave{m}$} for all $i$. Then
\begin{align*}
\wave{M} = \bigoplus_{\omega \in {\wt(M) }} \wave{M}_{\omega}.
\end{align*}
By Lemma \ref{fd_weight}, 
there is a unique {$\gamma \in \wt(\wave{M}) $}
such that {$\gamma \ge \omega$} for any  {$\omega \in \wt(\wave{M}) $}.

Choose a non-zero {$\wave{m} \in  \wave{M}_{\gamma} $}
and fix $ 1 \le j \le n-1$, 
set  {$\wave{m}_{j,p} = {\genF}_{j}^{(p)} \wave{m}$}.
Then by Proposition \ref{mp_shift}, 
if {$\gamma_{j}  \notin {v}^{\NN}$}, 
the last equation shows {$\wave{m}_{j,p} \ne 0$} for each  {$p$}, 
while the first equation shows  
{$\{ \wave{m}_{j,p} \where p = 1, 2, \cdots  \}$} 
is a set of  linearly independent vectors,
which contradicts to {$\mbox{dim}_{\wQv} \wave{M} < \infty$}.
Hence
there exist an integer $k > 0$ such that 
$\wave{m}_{j,k} \ne 0 $ and $\wave{m}_{j,k+1} = 0 $,
which means
\begin{align*}
0 
= {\genE}_{j} m_{j,k+1} 
=
		   \frac{   
		  	\gamma_{j}   {v}^{-k} 
		  	-  \gamma_{j}^{-1}  {v}^{k} 
		  	  }  {  {v} - {v}^{-1} }
		  	  m_{j,k} .
\end{align*}
and hence we have 
\begin{align*}
		\gamma_{j}   {v}^{-k} 
= \gamma_{j}^{-1}  {v}^{k} 
\end{align*}
and $\gamma_{j}  =   \pm {v}^{k} \in \pm{v}^{\NN} \subset \Qv$,
which means {$\gamma \in {(\Qv)}^{n}$},
and there is a vector {$ m \in M$} of weight {$\gamma$},
hence  {$ m \in M$} is a highest vector of highest  weight {$\gamma$} 
and {$M = \Uv m$}.
\end{proof}

\spaceintv
\section{The regular representation of   
twisted queer $q$-Schur superalgebra {$\SQvnrR$}}\label{regular}

By \cite{DGLW2025}, 
the twisted queer $q$-Schur superalgebra {$\SQvnrR$} is a left {$\Uv$}-module.
In this section,
we study the structure of {$\SQvnrR$} as {$\Uv$}-module.
We use {$P = {\Qv}^{\times }$} as above.

\begin{prop}
\cite[Lemma 2.3]{DLZ}
Let  $M$ be a polynomial weight supermodule and {$\lambda \in \wt(M)$}.
Then {$E_{h} M_{\lambda} = 0$} if {$\lambda_{h+1} = 0$},
and {$F_{h} M_{\lambda} = 0$} if {$\lambda_{h} = 0$}.
\end{prop}
\begin{rem}
The converse may not be correct. For example, 
let {$n=3, r=6$}, {$\lambda = (3,2,1)$}.
The Verma module {$M(\lambda) $} is a highest weight of {$\lambda$}, 
we have {$E_{1} M_{\lambda} = 0$} but {$\lambda_{2} \ne 0$}.
\end{rem}

\begin{prop}\label{module_poly_high}
Let {$M$} be an irreducible submodule of {$\SQvnrR$}, then 
{$M$} is polynomial weight module.
Furthermore, $M$ is a highest weight module of highest weight {${v}^{\omega} \in P^n$} 
with {$\omega_1 =  {r}$},   {$\omega_i = 0$} for all {$2 \le i \le n$}.  
\end{prop}
\begin{proof}
As  {$\SQvnrR = \tspan \{\Phi_{\Ad} \where {\Ad } \in \MNZ(n, r) \}$} is  finite dimensional, 
so is its submodule {$M$}.
 For any {$\lambda \in \CMN(n, r)$},
 \begin{align*}
 M_{\lambda} = \{  x = \sum  k_{\Xd } {\Phi}_{\Xd} \in M \where \ro(X) = \lambda, \  k_{\Xd} \in \Qv\}
 \end{align*}
  is the weight space of weight {${v}^{\lambda}$}.
Indeed,  by Theorem \ref{mulformAonPhi}\rm{(1)},
for any  {$\Xd \in \MNZ(n, r) $} satisfying {$\ro(X) = \lambda$}, 
 {${\genK}_{i} {\Phi}_{\Xd}  = {v}^{\lambda_i} {\Phi}_{\Xd}$}, 
hence {${\genK}_{i} m = {v}^{\lambda_i} m$}  
 and $M$ is a weight module.
 Furthermore,   Theorem \ref{mulformAonPhi}\rm{(2)} implies
 {${\genE}_{j} M_{\lambda} \ne 0$} if and only if {$\lambda_{j+1} \ne 0$}.
 Then $M$ is a weight module of highest weight {${v}^{\omega}$}
 with {$\omega = (r, 0, \cdots, 0) \in \CMN(n, r)$}.
\end{proof}

\begin{rem}
By Proposition \ref{module_poly_high},
for any  {$\lambda \in \CMN(n, r)$},
we use {$\lambda$} to denote  {${v}^{\lambda} = ({v}^{\lambda_1}, \cdots, {v}^{\lambda_n} )$} when it is viewed as a weight.
\end{rem}

Denote
\begin{align*}
1_r = \sum_{  \mu \in \CMN(n,r)  } {\Phi}_{\mu},  \qquad
1_{\SQvnR} = \sum_{r \ge 0} 1_r .
\end{align*}
It is verified  {$1_{\SQvnR} $} is the unit in {${\SQvnR} $}.
 Without confusion, we denote  {$1_{\SQvnR} $} as {$1$}.
Furthermore, we have
\begin{align*}
{\SQvnR} \cdot 1_r = 1_r \cdot {\SQvnR} = {\SQvnrR}.
\end{align*}

Recall the  superalgebra homomorphism  {$ \bs{\eta}_r : \USnv \rightarrow \SQvnrR$} and  the isomorphism   {$\bs{\xi}_n: \Uvq(n) \to \USnv$},
consider  the composition
\begin{align*}
	 \bs{\eta}_r \bs{\xi}_n  : \Uv & \longrightarrow  \SQvnrR .
\end{align*}
We denote
\begin{align*}
&  {\genK}_{i, r}^{\pm 1} = 	\bs{\eta}_r \bs{\xi}_n ({\genK}_{i}^{\pm 1})  , \qquad 
 {\genE}_{j,r}=  	\bs{\eta}_r \bs{\xi}_n ({\genE}_{j}) , \qquad 
{\genF}_{j, r} =  	\bs{\eta}_r \bs{\xi}_n  ({\genF}_{j}) , \\
&	 {\genK}_{\ol{i}, r} = \bs{\eta}_r \bs{\xi}_n ({\genK}_{\ol{i}}) , \qquad
	 {\genE}_{\ol{j},r} = \bs{\eta}_r \bs{\xi}_n ({\genE}_{\ol{j}}) , \qquad
	{\genF}_{\ol{j}, r} = \bs{\eta}_r \bs{\xi}_n({\genF}_{\ol{j}}) .
\end{align*}
 
Denote 
\begin{align*}
 \intds(X,k) =  \prod_{t=1}^{k} \frac{  {v}^{ 1 - t} X - {v}^{ - 1 + t} X^{-1}  }{ {v}^t - {v}^{-t} }.
\end{align*}
For any   {$\bs{k} \in \NN^n$},  set
\begin{align*}
\intdss(\bs{\genK}_r, \bs{k}) 
= \prod_{i=1}^{n} \intdss({\genK}_{i, r}, {k_i})
=\prod_{i=1}^{n}   \prod_{t=1}^{k_i} \frac{   {v}^{ 1 - t} {\genK}_{i, r}  - {v}^{ - 1 + t}  {\genK}_{i, r}^{-1}  }{ {v}^t - {v}^{-t} }.
\end{align*}

\begin{lem}\label{lem_bino1}
Let  $p,u \in \NN$ and $p>u$.
We have 
\begin{align*}
\prod_{t=1}^{u} 
		\frac{ {v}^{p - t } 	  - {v}^{-p+ t}  }
		{ {v}^{t} - {v}^{-t} } \in \ZZ[{v}, {v}^{-1}] .
\end{align*}
\end{lem}
\begin{proof}
Recall the notation from  \cite[IV.2 Equation (2.1)]{Kas}
\begin{align*}
(n)_{q} =  1 + {q} + {q}^2 + \cdots + {q}^{n-1} = \frac{{q}^{n} - 1}{{q} - 1},
\end{align*}
and  \cite[Proposition IV.2.1]{Kas} shows
\begin{align*}
		\frac{ \prod_{t=1}^{u} {( p - t)}_{q}}{  \prod_{t=1}^{u}  {(t)}_{q}} 
  =
	\frac{  {( p - 1)}_{q} {( p - 2)}_{q} \cdots {( p - u)}_{q} }
	{    {(1)}_{q}    {(2)}_{q} \cdots   {(u)}_{q}} 
= 
	\frac{  {( p - 1)}_{q} {( p - 2)}_{q} \cdots {( p - u)}_{q}  \times{( p -1 - u)!}_{q} }
	{  {(u)!}_{q} \times{( p -1 - u)!}_{q}} 
= {\left(\begin{matrix}{p-1}\\{u}\end{matrix}\right)}_{q}
 \in \ZZ[{q}] =   \ZZ[{v}^2] .
\end{align*}
As a consequence,
\begin{align*}
\prod_{t=1}^{u} 
		\frac{ {v}^{p - t } - {v}^{-p+ t}  }{ {v}^{t} - {v}^{-t} } 
& = 
\prod_{t=1}^{u} 
		\frac{ {v}^{p  } - {v}^{-p+2 t}  }{ {v}^{2t} - 1 }  \\
& = 
\prod_{t=1}^{u}  (
{v}^{-p} \cdot 
		\frac{ {v}^{ 2p  } - {v}^{2 t}  }{ {v}^{2t} - 1 }  ) \\
& = 
{v}^{-p u } \cdot 
\prod_{t=1}^{u} {v}^{2t}
\cdot 
\prod_{t=1}^{u} 
		\frac{ {v}^{ 2p - 2t } -  1}{ {v}^{2t} - 1 }  \\
& = 
{v}^{-p u } \cdot 
\prod_{t=1}^{u} {v}^{2t}
\cdot 
\prod_{t=1}^{u} 
		\frac{ {q}^{ p - t } -  1}{ {q}^{t} - 1 }  \qquad \mbox{( by {${q} = {v}^{2} $}) } \\
& = 
{v}^{-p u } \cdot 
\prod_{t=1}^{u} {v}^{2t}
\cdot 
		\frac{ \prod_{t=1}^{u} {( p - t)}_{q}}{  \prod_{t=1}^{u}  {(t)}_{q}}  
		\qquad
		\in \ZZ[{v}, {v}^{-1}] 
\end{align*}
and the result is proved.
\end{proof}
\begin{lem}\label{binorm_mul}
For any  {$1 \le i \le n$}, {$u \in \NN$} , {$\bs{k} \in \NN^n$} with {$\snorm{\bs{k}} \le r$} we have
\begin{align*}
&{\rm(1)} \quad \intdss({\genK}_{i, r},u) 
  =\sum_{\substack{  \\ \lambda \in \CMN(n,r) \\ \lambda_i  \ge p   }} 
	\prod_{t=1}^{u} 
		\frac{ {v}^{\lambda_i + 1 - t } 	  - {v}^{- \lambda_i - 1 + t}  }
		{ {v}^t - {v}^{-t} } {\Phi}_{(\lambda |O )} , \\
&{\rm(2)} \quad \intdss(\bs{\genK}_r, \bs{k}) 
 =\sum_{\substack{  \\ \lambda \in \CMN(n,r) \\ \lambda  \ge \bs{k}    }} 
 \prod_{i=1}^{n}   
	\prod_{t=1}^{k_i} 
		\frac{ {v}^{\lambda_i + 1 - t } 	  - {v}^{- \lambda_i - 1 + t}  }
		{ {v}^t - {v}^{-t} } {\Phi}_{(\lambda |O )} .
\end{align*}
where {${\lambda}  \ge  \bs{k}$} means {${\lambda}_i  \ge  {k}_i$} for all {$i$}.
Furthermore, we have 
{$ \intdss({\genK}_{i, r},u) , \intdss(\bs{\genK}_r, \bs{k}) \in \SQvnrRZ$}.
\end{lem}
\begin{proof}
By
\begin{align*}
{\genK}_{i, r}
=  \sum_{\substack{  \\ \lambda \in \CMN(n,r)   }} 
		{v}^{\lambda_i}  {\Phi}_{(\lambda |O )}, \qquad
{\genK}_{i, r}^{-1}
=  \sum_{\substack{  \\ \lambda \in \CMN(n,r)   }} 
		{v}^{- \lambda_i}  {\Phi}_{(\lambda |O )}, \qquad
\end{align*}
we have 
\begin{align*}
\frac{ {\genK}_{i, r} {v}^{  1 -  t} - {\genK}_{i, r}^{-1} {v}^{  - 1 + t} }{ {v}^t - {v}^{-t} } 
&=\frac{  \sum_{\substack{  \\ \lambda \in \CMN(n,r)   }} 
		{v}^{\lambda_i}  {\Phi}_{(\lambda |O )} \cdot{v}^{  1 -  t}
		 -   \sum_{\substack{  \\ \lambda \in \CMN(n,r)   }} 
		{v}^{- \lambda_i}  {\Phi}_{(\lambda |O )} \cdot {v}^{  - 1 + t} }
		 { {v}^t - {v}^{-t} }  \\
&=  \sum_{\substack{  \\ \lambda \in \CMN(n,r)   }} 
		\frac{ {v}^{\lambda_i + 1 - t } 	  - {v}^{- \lambda_i - 1 + t}  }
		{ {v}^t - {v}^{-t} } {\Phi}_{(\lambda |O )},
\end{align*}
and 
\begin{align*}
\intdss({\genK}_{i, r}, u) 
&=  \prod_{t=1}^{u} \frac{ {\genK}_{i, r} {v}^{ 1 - t} - {\genK}_{i, r}^{-1} {v}^{ - 1 + t} }{ {v}^t - {v}^{-t} } \\
&=  \prod_{t=1}^{u} 
	\sum_{\substack{  \\ \lambda \in \CMN(n,r)   }} 
		\frac{ {v}^{\lambda_i + 1 - t } 	  - {v}^{- \lambda_i - 1 + t}  }
		{ {v}^t - {v}^{-t} } {\Phi}_{(\lambda |O )}  \\
&=  
	\sum_{\substack{  \\ \lambda \in \CMN(n,r)   }} 
	\prod_{t=1}^{u} 
		\frac{ {v}^{\lambda_i + 1 - t } 	  - {v}^{- \lambda_i - 1 + t}  }
		{ {v}^t - {v}^{-t} } {\Phi}_{(\lambda |O )}  .
\end{align*}
Considering the fact that for {$p \in \NN$},  
\begin{align*}
&\prod_{t=1}^{u} 
		\frac{ {v}^{p - t } 	  - {v}^{-p+ t}  }
		{ {v}^{t} - {v}^{-t} } 
=
		\frac{ {v}^{p - 1 } 	  - {v}^{-p+ 1}  }
		{ {v}^1- {v}^{-1} } \times
		\frac{ {v}^{p - 2 } 	  - {v}^{-p+ 2}  }
		{ {v}^{2} - {v}^{-2} }  \times 
		\cdots  \times
		\frac{ {v}^{p - u } 	  - {v}^{-p+ u}  }
		{ {v}^{u} - {v}^{-u} }   .
\end{align*}
When {$p \le u$}, 
direct calculation shows
\begin{align*}
\prod_{t=1}^{u} 
		\frac{ {v}^{p - t } 	  - {v}^{-p+ t}  }
		{ {v}^{t} - {v}^{-t} }    = 0.
\end{align*}
For the case  {$p > u$}, Lemma \ref{lem_bino1}  implies
\begin{align*}
\prod_{t=1}^{u} 
		\frac{ {v}^{p - t } 	  - {v}^{-p+ t}  }
		{ {v}^{t} - {v}^{-t} } \in \fcZ .
\end{align*}
As a consequence, 
\begin{align*}
\intdss({\genK}_{i, r}, u) 
&=  
	\sum_{\substack{  \\ \lambda \in \CMN(n,r) \\ \lambda_i \ge u   }} 
	\prod_{t=1}^{u} 
		\frac{ {v}^{\lambda_i + 1 - t } 	  - {v}^{- \lambda_i - 1 + t}  }
		{ {v}^t - {v}^{-t} } {\Phi}_{(\lambda |O )} \in \SQvnrRZ, 
\end{align*}
and
\begin{align*}
\intdss(\bs{\genK}_r, \bs{k}) 
& = \prod_{i=1}^{n} \intdss({\genK}_{i, r}, {k_i}) \\
& = \prod_{i=1}^{n}   
	\sum_{\substack{  \\ \lambda \in \CMN(n,r) \\ \lambda_i \ge k_i   }} 
	\prod_{t=1}^{k_i} 
		\frac{ {v}^{\lambda_i + 1 - t } 	  - {v}^{- \lambda_i - 1 + t}  }
		{ {v}^t - {v}^{-t} } {\Phi}_{(\lambda |O )} \\
& =
\sum_{\substack{  \\ \lambda \in \CMN(n,r) \\ \lambda \ge \bs{k}  }} 
 \prod_{i=1}^{n}   
	\prod_{t=1}^{k_i} 
		\frac{ {v}^{\lambda_i + 1 - t } 	  - {v}^{- \lambda_i - 1 + t}  }
		{ {v}^t - {v}^{-t} } {\Phi}_{(\lambda |O )} \in \SQvnrRZ.
\end{align*}
\end{proof}

\begin{cor}\label{cor_binorm_mul}
For any  {$ \alpha \in \CMN(n, r)$}, we have
\begin{align*}
\intdss(\bs{\genK}_r, \alpha)   =	{\Phi}_{( \alpha | O )}.
\end{align*}
\end{cor}
\begin{proof}
Because {$ \lambda \in \CMN(n,r) $} and {$ {\lambda}  \ge \alpha $} imply
{$ {\lambda} = \alpha $}, 
then Lemma \ref{binorm_mul} leads
\begin{align*}
\intdss(\bs{\genK}_r, \alpha) 
&=
\sum_{\substack{  \\ \lambda \in \CMN(n,r) \\ \lambda \ge \alpha  }} 
 \prod_{i=1}^{n}   
	\prod_{t=1}^{{\alpha}_i} 
		\frac{ {v}^{\lambda_i + 1 - t } 	  - {v}^{- \lambda_i - 1 + t}  }
		{ {v}^t - {v}^{-t} } {\Phi}_{(\lambda |O)} \\
&=
 \prod_{i=1}^{n}   
	\prod_{t=1}^{{\alpha}_i} 
		\frac{ {v}^{\alpha_i + 1 - t } 	  - {v}^{- \alpha_i - 1 + t}  }
		{ {v}^t - {v}^{-t} } {\Phi}_{(\alpha |O )} .
\end{align*}
Observing
\begin{align*}
&\prod_{t=1}^{{\alpha}_i} 
		\frac{ {v}^{\alpha_i + 1 - t } 	  - {v}^{- \alpha_i - 1 + t}  }
		{ {v}^t - {v}^{-t} }  \\
&= 
		\frac{ {v}^{\alpha_i + 1 - 1 } 	  - {v}^{- \alpha_i - 1 + 1}  }
		{ {v} - {v}^{-1} }   \times
		\frac{ {v}^{\alpha_i + 1 - 2 } 	  - {v}^{- \alpha_i - 1 + 2}  }
		{ {v}^{2} - {v}^{-2} }  \times
		\cdots \times
		\frac{ {v}^{\alpha_i + 1 - \alpha_i } 	  - {v}^{- \alpha_i - 1 + \alpha_i}  }
		{ {v}^{\alpha_i} - {v}^{-\alpha_i} }   \\
&= 
		\frac{ {v}^{\alpha_i} 	  - {v}^{- \alpha_i }  }
		{ {v} - {v}^{-1} }  \times
		\frac{ {v}^{\alpha_i -1 } 	  - {v}^{- \alpha_i +1}  }
		{ {v}^{2} - {v}^{-2} }    \times
		\cdots  \times
		\frac{ {v}	  - {v}^{-1 }  }
		{ {v}^{\alpha_i} - {v}^{-\alpha_i} }   \\
& = 1,
\end{align*}
then we have
\begin{align*}
\intdss(\bs{\genK}_r, \alpha) 
&= {\Phi}_{(\alpha |O )} .
\end{align*}

\end{proof}

\begin{prop}\label{mod_decp3}
There is a  {$\Uv$}-module decomposition 
\begin{align*}
\SQvnrR = \bigoplus_{\mu \in \CMN(n, r)} \SQvnrR^{\mu},
\end{align*}
where each submodule
\begin{align*}
 \SQvnrR^{\mu} 
 = \tspan \{ 	{\Phi}_{\Ad} 
		\where \Ad \in \MNZ(n, r), \co(A)=\mu  
\}.
\end{align*}
Furthermore,  each submodule $ \SQvnrR^{\mu}$ has a weight space decompisition 
\begin{align*}
 \SQvnrR^{\mu} = \bigoplus_{\lambda \in \CMN(n, r)} \SQvnrR_{\lambda}^{\mu}
\end{align*}
with  each  weight space 
\begin{align*}
\SQvnrR_{\lambda}^{\mu}  = \tspan \{ 	{\Phi}_{\Ad} 
		\where \Ad \in \MNZ(n, r), \ro(A) = \lambda, \co(A)=\mu  
\} 
\end{align*}
of weight $\lambda$.
\end{prop}
\begin{proof}
For any $m \in {\SQvnrR}$, 
It is trivial 
\begin{align*}
m = \sum_{\mu \in \CMN(n, r)} \sum_{\co(A) = \mu} k_{\Ad} \Phi_{\Ad}  
\end{align*}
where  {$ {k_{\Ad}} \in \Qv$}  for each {$  {\Ad} \in \MNZ(n, r)$}, 
and  there is  a subspace decomposition
\begin{align*}
{\SQvnrR} = \bigoplus_{\mu \in \CMN(n, r)} {\SQvnrR}^{\mu}.
\end{align*}
For  any {$X \in \Uv$},
 and {$ {\Ad} \in \MNZ(n, r) $}  with {$\co(A) = \mu$},
 Theorem \ref{mulformAonPhi} and Remark \ref{induct_all} imply
 \begin{align*}
X \Phi_{\Ad}
&=\sum_{ \substack{\co(B) = \mu  } } k_{\Bd} {\Phi}_{\Bd} \in {\SQvnrR}^{\mu}.
\end{align*}
Hence  each {${\SQvnrR}^{\mu}$}  is a submodule of ${\SQvnrR}$,
and the module decomposition is proved.

Furthermore, for any  {$m^{\mu} \in {\SQvnrR}^{\mu}$},
it is clear
\begin{align*}
  m^{\mu}  
&=
\sum_{\lambda \in \CMN(n, r)}
\sum_{\substack{
\ro(C) = \lambda, \\
\co(C) = \mu
}} 
 k_C {\Phi}_{C} ,
\end{align*}
with each {$ {k_C} \in \Qv$}.
For each {$ {\Phi}_{C}$} with {$\ro(C) = \lambda $} in the last equation, 
Theorem \ref{mulformAonPhi}\rm{(1)} implies
\begin{align*}
{\genK}_{i} {\Phi}_{C}  
	& = 
		{v}^{\lambda_i } 
		{\Phi}_{C} ,
\end{align*}
and the weight decomposition of $ {\SQvnrR}^{\mu}$ is now proved.
\end{proof}

Applying Proposition \ref{wt_shift} and Proposition \ref{mod_decp3},  
we have the following.
\begin{cor}\label{irr_co}
Assume $\lambda, \mu \in \CMN(n, r)$ and {$\lambda = (r, 0, \cdots, 0) $}.
Then each weight vector in   {$\SQvnrR_{\lambda}^{\mu}$}  is a highest weight vector.
\end{cor}

\begin{prop}\label{f_order}
Assume {$\lambda, \mu \in \CMN(n, r)$},  {$\lambda_{i+1} = 0$},  
and  assume  {$m \in {\SQvnrR}^{\mu}_{\lambda}$} is a non-zero vector.
Then for any {$k \in \NN^+$}, 
\begin{align*}
	{\genF}_{i} {\genF}_{i+1}  {\genF}_{i}^{k} m = {\cft}_{k} {\genF}_{i+1}   {\genF}_{i}^{k+1} m ,
\end{align*}
	where each {${\cft}_{k} \in \Qv$}.
\end{prop}
\begin{proof}
We prove it by induction.
For the non-zero {$m \in {\SQvnrR}^{\mu}_{\lambda}$},
the second relation   in (QQ6) (by setting $j=i+1$)  leads
\begin{align*}
{\genF}_{i}^2 {\genF}_{i+1} m - ( {v} + {v}^{-1} ) {\genF}_{i} {\genF}_{i+1} {\genF}_{i} m + {\genF}_{i+1}  {\genF}_{i}^2 m = 0. 
\end{align*}
Considering   {$\lambda_{i+1} = 0$} implies 	{${\genF}_{i+1} m = 0$},
it follows
\begin{align*}
( {v} + {v}^{-1} ) {\genF}_{i} {\genF}_{i+1} {\genF}_{i} m =  {\genF}_{i+1}  {\genF}_{i}^2 m  
\end{align*}
and
\begin{align*}
	{\genF}_{i} {\genF}_{i+1} ({\genF}_{i} m) =  \frac{1}{( {v} + {v}^{-1} )}  {\genF}_{i+1}  {\genF}_{i} ( {\genF}_{i} m ) ,
\end{align*}
which means the result holds for the case $k=1$.
Assume it holds for any $k \in \NN^+$, 
then by multiplying ${\genF}_{i}^{k} m$ on the right side of each summand of the second equation of (QQ6), 
we conclude
\begin{align*}
0 
&= {\genF}_{i}^2 {\genF}_{i+1} ( {\genF}_{i}^{k} m )  - ( {v} + {v}^{-1} ) {\genF}_{i} {\genF}_{i+1} {\genF}_{i} ( {\genF}_{i}^{k} m )  + {\genF}_{i+1}  {\genF}_{i}^2 ( {\genF}_{i}^{k} m )  \\
&= {\genF}_{i} \cdot {\genF}_{i} {\genF}_{i+1} ( {\genF}_{i}^{k} m )  - ( {v} + {v}^{-1} ) {\genF}_{i} {\genF}_{i+1} {\genF}_{i} ( {\genF}_{i}^{k} m )  + {\genF}_{i+1}  {\genF}_{i}^2 ( {\genF}_{i}^{k} m )  \\
&= {\genF}_{i} \cdot {\cft}_{k} {\genF}_{i+1} {\genF}_{i} ( {\genF}_{i}^{k} m )  - ( {v} + {v}^{-1} ) {\genF}_{i} {\genF}_{i+1} {\genF}_{i} ( {\genF}_{i}^{k} m )  + {\genF}_{i+1}  {\genF}_{i}^2 ( {\genF}_{i}^{k} m )  \\
&= {( {\cft}_{k} - {v} - {v}^{-1} ) } {\genF}_{i} {\genF}_{i+1} {\genF}_{i} ( {\genF}_{i}^{k} m )  + {\genF}_{i+1}  {\genF}_{i}^2 ( {\genF}_{i}^{k} m ) ,
\end{align*}
which implies
\begin{align*}
{( {v} + {v}^{-1} - {\cft}_{k} ) } {\genF}_{i} {\genF}_{i+1} {\genF}_{i} ( {\genF}_{i}^{k} m )  =  {\genF}_{i+1}  {\genF}_{i}^2 ( {\genF}_{i}^{k} m ) 
\end{align*}
and
\begin{align*}
{\genF}_{i} {\genF}_{i+1} ( {\genF}_{i}^{k+1} m )  =  {( {v} + {v}^{-1} - {\cft}_{k} ) }^{-1} {\genF}_{i+1}  {\genF}_{i} ( {\genF}_{i}^{k+1} m ),
\end{align*}
hence the equation holds for $k+1$ and for any positive number.
\end{proof}

By Proposition \ref{f_order}, 
we conclude the following corollary by induction.
\begin{prop}\label{f_order2}
Assume {$\lambda, \mu \in \CMN(n, r)$},  {$\lambda_{i+1} = 0$},
{$p, k \in \NN^+$},
and {$m \in {\SQvnrR}^{\mu}_{\lambda}$} is a non-zero vector,
{$\lambda_{i } \ge k$}, and {$p \le k$}. 
Then we have
\begin{align*}
	{\genF}_{i} {\genF}_{i+1}^{p} {\genF}_{i}^{k} m    = {\cft}_{p, k} {\genF}_{i+1}^{p}  {\genF}_{i}^{k+1} m  ,
\end{align*}
where {${\cft}_{p, k} \in \Qv$}.
\end{prop}
\begin{proof}
The case $p=1$ is Proposition \ref{f_order}.
By replacing {$i, j$} with {$i+1, i$} respectively for the second equation in (QQ6), 
we have  
\begin{equation}\label{qq6shift}
\begin{aligned}
	{\genF}_{i+1}^2 {\genF}_{i} - ( {v} + {v}^{-1} ) {\genF}_{i+1} {\genF}_{i} {\genF}_{i+1} + {\genF}_{i}  {\genF}_{i+1}^2 = 0.
\end{aligned}
\end{equation}
Multiplying   ${\genF}_{i}^{k} m$ on the right side of each summand of \eqref{qq6shift}, 
we obtain
\begin{align*}
	{\genF}_{i+1}^2 {\genF}_{i} {\genF}_{i}^{k} m  - ( {v} + {v}^{-1} ) {\genF}_{i+1} {\genF}_{i} {\genF}_{i+1} {\genF}_{i}^{k} m  + {\genF}_{i}  {\genF}_{i+1}^2 {\genF}_{i}^{k} m = 0.
\end{align*}
As a consequence, it follows
\begin{align*}
 {\genF}_{i}  {\genF}_{i+1}^2 {\genF}_{i}^{k} m 
 &=
 - {\genF}_{i+1}^2 \cdot {\genF}_{i} {\genF}_{i}^{k} m  + ( {v} + {v}^{-1} ) {\genF}_{i+1} \cdot  {\genF}_{i} {\genF}_{i+1} {\genF}_{i}^{k} m  \\
 &=
 - {\genF}_{i+1}^2  {\genF}_{i}^{k+1} m  + ( {v} + {v}^{-1} ) {\genF}_{i+1} \cdot {\cft}_{k} {\genF}_{i+1}   {\genF}_{i}^{k+1} m   \\
 &=  ( {v}   {\cft}_{k}  + {v}^{-1}   {\cft}_{k} - 1) {\genF}_{i+1}^2     {\genF}_{i}^{k+1} m ,
\end{align*}
which means the result holds for the case $p=2$.

Assume the result holds for the cases $p = a-1$ and   $p = a $.
Multiplying   $ {\genF}_{i+1}^{a-1} {\genF}_{i}^{k} m $ on the right side of each summand of \eqref{qq6shift}, 
we obtain
\begin{align*}
	{\genF}_{i+1}^2 {\genF}_{i} \cdot {\genF}_{i+1}^{a-1} {\genF}_{i}^{k} m   - ( {v} + {v}^{-1} ) {\genF}_{i+1} {\genF}_{i} {\genF}_{i+1} \cdot {\genF}_{i+1}^{a-1} {\genF}_{i}^{k} m   + {\genF}_{i}  {\genF}_{i+1}^2 \cdot  {\genF}_{i+1}^{a-1} {\genF}_{i}^{k} m = 0.
\end{align*}
then 
 \begin{align*}
	{\genF}_{i} {\genF}_{i+1}^{a+1} {\genF}_{i}^{k} m   
&=
	{\genF}_{i} {\genF}_{i+1}^{2}  \cdot   {\genF}_{i+1}^{a-1} {\genF}_{i}^{k} m  \\
&=-	{\genF}_{i+1}^2 \cdot {\genF}_{i} {\genF}_{i+1}^{a-1} {\genF}_{i}^{k} m   
	+ ( {v} + {v}^{-1} ) {\genF}_{i+1} \cdot {\genF}_{i}  {\genF}_{i+1}^{a} {\genF}_{i}^{k} m   \\
&=-	{\genF}_{i+1}^2 \cdot  {\cft}_{a-1, k} {\genF}_{i+1}^{a-1}  {\genF}_{i}^{k+1} m   
	+ ( {v} + {v}^{-1} ) {\genF}_{i+1}  \cdot {\cft}_{a, k} {\genF}_{i+1}^{a}  {\genF}_{i}^{k+1} m \\
&=   
	  ( -	  {\cft}_{a-1, k}  +  {v}  {\cft}_{a, k}  + {v}^{-1}  {\cft}_{a, k}   )    {\genF}_{i+1}^{a+1}  {\genF}_{i}^{k+1} m,
\end{align*}
which means the case {$p=a+1$} is proved, and hence we have proved the result for all {$p \in \NN^+$}.
\end{proof}

Recall that {${\bs{\alpha}}_{j} = {\bs{\ep}}_{j} - {\bs{\ep}}_{j+1}$}. 

For any {$  \Phi_{\Ad} \in \SQvnrR$},
denote $\psi( \Phi_{\Ad} ) $ to be the maximal index of nonzero entry of $\ro(A)$.
In general, denote  
{$\psi(\sum {k_{\Ad}} \Phi_{\Ad}) = \max \{ \psi(A) \where k_{\Ad} \ne 0\}$}.

\begin{thm}\label{thm_module_structure}
Assume {$r \ge 2$},    {$\omega=(r, 0, \cdots, 0) \in \CMN(n, r)$},
let $M$ be an irreducible submodule of   {$\SQvnrR$}, 
and {$\mu \in \CMN(n, r)$} satisfies {$M = M^{\mu}$}.
Then \\
{\rm(i)}	
	$M$ is a highest weight module with highest weight {$\omega$},
	and there is a vector {$m \in M_{\omega}$}  such that  {$M = \Uv m$}; \\
{\rm(ii)} {$M_{\omega} = \tspan\{m, K_{\ol{1}} m\} $}; \\
{\rm(iii)} 
for any  {$\lambda \in \wt(M)$} with {$\lambda - {\bs{\alpha}}_{i} \in \wt(M)$},
we have  {$ M_{\lambda - {\bs{\alpha}}_i} = {\genF}_{i} M_{\lambda} $}.
\end{thm}
\begin{proof}
We choose a sequence of vectors {$\{ m_0, m_1, m_2, \cdots \} $} as:
\begin{enumerate}
\item
let $m_0 = \sum {k_{\Ad}} \Phi_{\Ad}$ be any nonzero vector of $M$,
\item
if $t = \psi(m_{j}) > 1$, set $m_{j+1} = {\genE}_{t - 1} m_{j}$,
\item
if $t = \psi(m_{j}) = 1$, set $m_{j+1} = 0$.
\end{enumerate}
It is clear that there are only finite number of nonzero {$m_j$}.
Let $m = m_{k}$ such that {$m_k \ne 0$} and  {$m_{k+1} = 0$},
then ${\genE}_{j} m = 0$,  ${\genE}_{\ol{j}} m = 0$  for all {$1 \le j \le n-1$}.
Theorem \ref{mulformAonPhi}\rm(2) shows that {$m = \sum k_{\Ad} {\Phi}_{\Ad}$} 
and each {$\Phi_{\Ad}$} satisfies {$\ro(A) = \omega$}.
Hence $m$ is a highest weight vector of highest weight $\omega$.

For any   {$\lambda \in \wt(M)$},
we define {$L(m, \lambda)$} by induction as:
\begin{enumerate}
\item
for the case {$\lambda = \omega$}, set {$L(m, \omega) = \tspan\{m, K_{\ol{1}} {m} \}$};
\item
 for any {$\lambda$} with {$\lambda - {\bs{\alpha}}_{i} \in \wt(M)$} 
 (or {$\lambda_i > 0$} equivalently),
set
{$
	L(m, \lambda - {\bs{\alpha}}_{i}) = \Uv^0 {\genF}_{i} L(m, \lambda) .
$}
\end{enumerate}
The last equation of  \rm{(QQ1)} shows that {$L(m, \omega) $} is a $\Uv^0$-module.

For any {$1 \le i \le n-1$} with {$\lambda_{i+1} = 0$},
and any {$p,k \in \NN^+$},
applying Proposition \ref{f_order2},
 we have 
\begin{align*}
	{\genF}_{i} {\genF}_{i+1}^{p} {\genF}_{i}^{k+1}L(m, \lambda)    
	  = {\cft}_{p, k} {\genF}_{i+1}^{p}  {\genF}_{i}^{k+1} L(m, \lambda)   
\end{align*}
and hence the notation {$L(m, \lambda)$} is well-defined.
It is easy to verify {$L(m, \lambda) \subseteq M_{\lambda}$}
by applying Theorem \ref{mulformAonPhi}.

We claim that for any {$\lambda \in \CMN(n, r)$},
we have 
\begin{align}
{\genE}_{i} {\genF}_{i} L(m, \lambda) \subseteq L(m, \lambda)  , \qquad
{\genF}_{i} {\genE}_{i} L(m, \lambda) \subseteq L(m, \lambda) .
\label{ind_k_L0}
\end{align}
It suffices to prove 
\begin{align}
{\genE}_{i} {\genF}_{i} {\genF}_{i}^{k} L(m, \gamma) \subseteq L(m, \gamma - k {\bs{\alpha}}_{i} )  , \qquad
{\genF}_{i} {\genE}_{i} {\genF}_{i}^{k}  L(m, \gamma) \subseteq L(m, \gamma - k{\bs{\alpha}}_{i} ),
\label{ind_k_L1}
\end{align}
for any $k \in \NN$ and {$\gamma \in \CMN(n, r)$} satisfying {$k \le\gamma_{i} $},  {$\gamma_{i+1} = 0$}.
Choose any {$w \in L(m, \gamma)$}, we have {${\genE}_{i} w = 0$}
and hence 
\begin{align*}
{\genE}_{i} {\genF}_{i} w 
&=  {\genF}_{i} {\genE}_{i} w
+  \frac{{\genK}_{i} {\genK}_{i+1}^{-1} - {\genK}_{i}^{-1}{\genK}_{i+1}}{{v} - {v}^{-1}} w  \\
&=   \frac{{\genK}_{i} {\genK}_{i+1}^{-1} - {\genK}_{i}^{-1}{\genK}_{i+1}}{{v} - {v}^{-1}} w \in  L(m, \gamma),
\end{align*}
which means
\begin{align*}
{\genE}_{i} {\genF}_{i}  L(m, \gamma) \subset L(m, \gamma)  , \qquad
{\genF}_{i} {\genE}_{i}  L(m, \gamma) \subset L(m, \gamma) .
\end{align*}
When $k = 1$,
we have
\begin{align*}
 & {\genF}_{i}  {\genE}_{i} {\genF}_{i} w 
 \in  L(m, \gamma  - {\bs{\alpha}}_{i} ),\\
&{\genE}_{i} {\genF}_{i} {\genF}_{i}  w 
=  {\genF}_{i} {\genE}_{i} {\genF}_{i}  w
+  \frac{{\genK}_{i} {\genK}_{i+1}^{-1} - {\genK}_{i}^{-1}{\genK}_{i+1}}{{v} - {v}^{-1}} {\genF}_{i}  w 
 \in  L(m, \gamma  - {\bs{\alpha}}_{i} ).
\end{align*}
By induction on $k$,
assume \eqref{ind_k_L1} holds for  {$k$},
then
\begin{align*}
 & {\genF}_{i}  {\genE}_{i} {\genF}_{i}^{k+1} w \in 
 		{\genF}_{i}   L(m, \gamma  - k {\bs{\alpha}}_{i} )
 		 =   L(m, \gamma  - (k+1) {\bs{\alpha}}_{i} )
 		,\\
&{\genE}_{i} {\genF}_{i} {\genF}_{i}^{k+1}  w 
=  {\genF}_{i} {\genE}_{i} {\genF}_{i}^{k+1}   w
+  \frac{{\genK}_{i} {\genK}_{i+1}^{-1} - {\genK}_{i}^{-1}{\genK}_{i+1}}{{v} - {v}^{-1}} {\genF}_{i}^{k+1}   w 
 \in  L(m, \gamma  - (k+1) {\bs{\alpha}}_{i} ) 
\end{align*}
and hence \eqref{ind_k_L1} holds for any  {$k \le\gamma_{i} $}.
As a consequence, 
  for {$\lambda \in \wt(M)$} with {$\lambda + {\bs{\alpha}}_{i} \in \wt(M)$} 
we have
\begin{align}
{\genE}_{i} L(m, \lambda) \subseteq L(m, \lambda + {\bs{\alpha}}_{i} ) .
\label{ind_k_L2}
\end{align}

Referring to \cite{DLZ},
by relations in (QQ3),  we have 
\begin{align*}
&   {\genE}_{\ol{i}}= {\genK}_{\ol{i}} {\genE}_{i}   {\genK}_{i} - {v} {\genE}_{i} {\genK}_{\ol{i}}  {\genK}_{i} , \qquad 
   {\genF}_{\ol{i}}  = - {\genK}_{\ol{i}} {\genF}_{i} {\genK}_{i}^{-1}+ {v} {\genF}_{i} {\genK}_{\ol{i}}{\genK}_{i}^{-1}.
\end{align*}
The relations in (QQ4) lead
\begin{align*}
&  {\genK}_{\ol{i+1}}  
= - {\genE}_{i} {\genF}_{\ol{i}} {\genK}_{i}+ {\genF}_{\ol{i}} {\genE}_{i}{\genK}_{i} +  {\genK}_{i+1}^{-1} {\genK}_{\ol{i}}  {\genK}_{i} , 
\end{align*}
By induction, 
  for {$\lambda \in \wt(M)$} with {$\lambda + {\bs{\alpha}}_{i} \in \wt(M)$} ,
we have
\begin{align}
{\genE}_{\ol{i}} L(m, \lambda) \subseteq L(m, \lambda + {\bs{\alpha}}_{i} ), 
\label{ind_k_L3}
\end{align}
and   for {$\lambda \in \wt(M)$} with {$\lambda - {\bs{\alpha}}_{i} \in \wt(M)$} ,
we have
\begin{align}
{\genF}_{\ol{i}} L(m, \lambda) \subseteq L(m, \lambda - {\bs{\alpha}}_{i} ).
\label{ind_k_L4}
\end{align}
By definition,  for any {$1 \le i \le n$} and  {$\lambda \in \wt(M)$}, we have
\begin{align}
{\genK}_{\ol{i}} L(m, \lambda) \subseteq L(m, \lambda ).
\label{ind_k_L5}
\end{align}

Set 
	{$	L=	\bigoplus_{\lambda \in \CMN(n, r)} L(m, \lambda) $}.
Summerizing  
\eqref{ind_k_L2},  \eqref{ind_k_L3},  \eqref{ind_k_L4}, \eqref{ind_k_L5}  
and the definition of {$L(m, \lambda)$},
we know {$L$} is a {$\Uv$}-module.
As {$L \subset M$} and $M$ is an irreducible {$\Uv$}-module, 
we have {$L = M$} and $M$ is generated by the highest weight vector $m$, 
which means \rm(i) is proved completely.

Furthermore, since {$L(m, \lambda) \subset M_{\lambda}$}, 
we have {$L(m, \lambda)= M_{\lambda}$} and \rm{(iii)} is proved.
Specially,   {$M_{\omega} = L(m, \omega) = \tspan\{m, K_{\ol{1}} m\} $},  then  {\rm(ii)} is proved.
\end{proof}

\begin{prop}\label{prop_module_top_structure}
Assume {$r \ge 1$}, 
$M$ is an irreducible submodule of {$\SQvnrR$}
with highest weight {$\omega = (r, 0, \cdots, 0) \in \CMN(n, r)$}.
Then    {$ \dim M_{\omega} = 2$} when {$r > 1$},
 {$ \dim M_{\omega} = 1$}  when {$r = 1$}.
\end{prop}
\begin{proof}
By Theorem \ref{thm_module_structure},  for any non-zero  {$m \in M_{\omega}$},
we have {$M_{\omega} = \tspan\{m, K_{\ol{1}} m\} $},
then {$\dim M_{\omega} \le 2$}.

If {$\dim M_{\omega} =1$},
then for any non-zero {$m \in  M_{\omega} $},
there exists {$\cft \in \Qv$} such that 
{$ K_{\ol{1}} m= {\cft} m $} and $  K_{\ol{1}}^2 m = {\cft}^2 m $.
Assume {$m = \SEE{m} + \SOE{m}$},
with
 {$\SEE{m} = \sum_{A \in \SE{J}} k_{\Ad} \Phi_{\Ad}$}, 
 {$\SOE{m} = \sum_{A \in \SO{J}} k_{\Ad} \Phi_{\Ad}$}, 
each {$k_{\Ad} \in \Qv$}, 
 and {$\SE{J} = \{ {\Ad} \in \MNZ(n, r) \where \parity{\Ad} = 0\}$}, 
 {$\SO{J} = \{ {\Ad} \in \MNZ(n, r) \where \parity{\Ad} = 1\}$}.
Then by the last equation in  (QQ1)  and $\genK_1 m = {v}^r m$, we have
\begin{align*}
  K_{\ol{1}}^2 m 
&=      \frac{  {v}^{2r} -  {v}^{-2r} }{{v}^2 - {v}^{-2}}  m ,
\end{align*}
and 
\begin{align*}
 {\cft} &=   \pm \sqrt{   \frac{  {v}^{2r} -  {v}^{-2r} }{{v}^2 - {v}^{-2}}  } \in \Qv, 
\end{align*}
hence we have {$r = 1$} and {$\cft = \pm 1$}.
In this case,  by 
\begin{align*}
  K_{\ol{1}} ( \SEE{m} + \SOE{m})= {\cft}  \SEE{m} + {\cft}  \SOE{m} ,
\end{align*}
comparing the even and odd parts, we have
\begin{align*}
  K_{\ol{1}} \SEE{m}  =  {\cft}  \SOE{m} ,\qquad
  K_{\ol{1}} \SOE{m} = {\cft}  \SEE{m}   ,
\end{align*}
then we have 
{$ m = \SEE{m} + K_{\ol{1}} \SEE{m} $}
or
{$ m = \SEE{m} - K_{\ol{1}} \SEE{m} $}.

When {$r > 1$}, 
$  K_{\ol{1}}^2 m = {\cft}^2 m $ implies a contradiction that 
\begin{align*}
 {\cft} &=   \pm \sqrt{   \frac{  {v}^{2r} -  {v}^{-2r} }{{v}^2 - {v}^{-2}}  } \notin \Qv, 
\end{align*}
and hence {$\dim M_{\omega} = 2$}.
\end{proof}

We define the set
\begin{align*}
 \CMN(n, r_0| r_1) 
 	=\{ (\SE{\lambda} | \SO{\lambda})  \where  
 				\SE{\lambda} \in \CMN(n, r_0), 
 				\SO{\lambda} \in \CMN(n, r_1),
 				  0 \le \SO{\lambda}_i \le 1 \mbox{ for any }1 \le i \le n  
 				\}.
\end{align*}
For any {$\lambda \in \CMN(n, r_0|r_1)$} and {$1 \le i \le n$},
we denote {$\lambda_i = \SE{\lambda}_i  + \SO{\lambda}_i $}.
The parity of {$\lambda$} is defined to be {$\parity{\lambda} \equiv r_1 (\mathrm{mod} 2)$}.

Assume  {$\mu \in \CMN(n, r)$},
we denote  
\begin{align*}
J_{\mu} = \{  \lambda \in \CMN(n, r_0| r_1) \where r_0 + r_1 = r, \ \SE{\lambda} + \SO{\lambda} = \mu\}.
\end{align*}
It is trivial that {$\# J_{\mu} = 2^{d(\mu)}$},
where {$d(\mu) = \# \{ 1 \le i \le n \where {\mu}_{i} > 0\}$}.

\begin{thm}\label{mod_decomp_irr}
There is  {$\Uv$}-module decomposition
\begin{align*}
{\SQvnrR} = \bigoplus_{\mu \in \CMN(n, r)} \SQvnrR^{\mu},
\end{align*}
each submodule {$\SQvnrR^{\mu}$} is a highest weight module over {$\Uv$},
with highest weight {${v}^{{\omega}}$},  
where {${\omega} = (r, 0, \cdots, 0) \in \CMN(n, r)$}.
And each {$\SQvnrR^{\mu}$} could be decomposed into the  direct sum of irreducible submodules:
\begin{align*}
{\rm(1)  } \qquad & \SQvnrR^{\mu} 
=		( \Uv ( {\Phi}_{A_{\mu}} + K_{\ol{1}} {\Phi}_{A_{\mu}} ) )
				\bigoplus 
		( \Uv ( {\Phi}_{A_{\mu}} - K_{\ol{1}} {\Phi}_{A_{\mu}} ) ) , \qquad \mbox{ when } r=1;\\
{\rm(2) } \qquad &
\SQvnrR^{\mu}
= \bigoplus_{ \lambda \in J_{\mu},  \  \parity{\lambda} = 0} 
\Uv   {\Phi}_{A_{\lambda}} ,\qquad
\mbox{ when }  r>1,
\end{align*}
where {$A_{\mu} = (\sum_{i=1}^{n} {\mu}_{i} E_{1, i} | O)$} 
for each {$\mu \in \CMN(n, r)$},
{$A_{\lambda} = (\sum_{i=1}^{n} \SE{\lambda}_{i} E_{1, i} | \sum_{i=1}^{n} \SO{\lambda}_{i} E_{1, i})$} for each {$ \lambda \in J_{\mu}$}.
In the meantime, the decomposition of {${\SQvnrR}$} as {$\Uv$}-modules 
is also a complete decomposition of the regular representation of  {${\SQvnrR}$}.
\end{thm}
\begin{proof}
By Proposition \ref{mod_decp3}
and Theorem \ref{thm_module_structure},
for any {$\mu \in \CMN(n, r)$}, 
we only need to discuss the decomposition of {$ \SQvnrR^{\mu}$} as {$\Uv$}-module.

For each {$\lambda \in J_{\mu}$}, 
we denote {$A_{\lambda} = (\sum_{i=1}^{n} \SE{\lambda}_{i} E_{1, i} | \sum_{i=1}^{n} \SO{\lambda}_{i} E_{1, i})$}.
Then the highest weight space of {$ \SQvnrR^{\mu}$}  is 
\begin{align*}
\SQvnrR^{\mu}_{{\omega}} 
= {\tspan}  \{ {\Phi}_{A_{\lambda}} \where \lambda \in J_{\mu} \}.
\end{align*}
For any {$ \lambda \in J_{\mu} $},  
denote  {$L(A_{\lambda}) = \Uv  {\Phi}_{A_{\lambda}} $}.
By Theorem \ref{thm_module_structure} 
and Proposition \ref{prop_module_top_structure}, 
when {$r = 1$}, 
{$( {\Phi}_{A_{\lambda}} \pm K_{\ol{1}} {\Phi}_{A_{\lambda}} )$}
are the eigenvectors of {$K_{\ol{1}}$},
hence $L(A_{\lambda})$ could be decomposed into a direct sum of two irreducible submodules:
\begin{align*}
L(A_{\lambda}) = {\Uv} ( {\Phi}_{A_{\lambda}} + K_{\ol{1}} {\Phi}_{A_{\lambda}} ) 
	\bigoplus {\Uv} ({\Phi}_{A_{\lambda}} - K_{\ol{1}} {\Phi}_{A_{\lambda}}).
\end{align*}
When {$r > 1$}, $L(A_{\lambda})$ is irreducible, 
and 
\begin{align*}
{L(A_{\lambda})}_{{\omega}} = \tspan \{    {\Phi}_{A_{\lambda}} ,  K_{\ol{1}} {\Phi}_{A_{\lambda}} \}, \qquad
\sum_{ \lambda \in J_{\mu},  \  \parity{\lambda} = 0}  \dim {L(A_{\lambda})}_{{\omega}}
 = \dim \SQvnrR^{\mu}_{{\omega}} .
\end{align*}
Hence we have module decomposition
\begin{align*}
\SQvnrR^{\mu}
= \bigoplus_{ \lambda \in J_{\mu},  \  \parity{\lambda} = 0} 
\Uv   {\Phi}_{A_{\lambda}} ,
\end{align*}
where each {$\Uv   {\Phi}_{A_{\lambda}} $} is an irreducible  {$\Uv$}-module.
\end{proof}
\spaceintv
\bibliographystyle{unsrt}


\iftrue

\fi
\end{document}